\documentclass[a4paper,11pt,reqno]{amsart}

\usepackage[english]{babel}
\usepackage{amsmath}
\usepackage{amsfonts}
\usepackage{amsthm}
\usepackage{float}
\usepackage{stmaryrd}
\usepackage{graphics}
\usepackage{graphicx}
\usepackage{subfig}
\usepackage{datetime}
\usepackage{color}
\newcommand{\Hloc}{H^1_{\textnormal{loc}}}

\renewcommand{\d}{\,\mathrm{d}}
\newcommand{\de}{\partial}

\newcommand{\mc}{\mathcal}

\def\enne{\mathbb{N}}
\def\zeta{\mathbb{Y}}

\def\erre{\mathbb{R}}

\renewcommand{\to}{\rightarrow}

\numberwithin{equation}{section}
\newtheorem{thm}{Theorem}[section]
\newtheorem{defi}[thm]{Definition}
\newtheorem{prop}[thm]{Proposition}
\newtheorem{lemma}[thm]{Lemma}
\newtheorem{cor}[thm]{Corollary}

\theoremstyle{definition}
\begingroup
\newtheorem{rmk}[thm]{Remark}

\endgroup

\theoremstyle{remark}
\begingroup
\endgroup

\begin{document}
	
	\author[Filippo Riva \and Lorenzo Nardini]{Filippo Riva$^*$ \and Lorenzo Nardini$^2$}
	\title[Existence and uniqueness for a damped debonding model]{Existence and uniqueness of dynamic evolutions\\for a one--dimensional debonding model with damping}
	\thanks{{$^*$ SISSA, \textsc{Via Bonomea, 265, 34136, Trieste, Italy.}\\ \hspace*{0.62cm}\textit{E-mail address}:\textsf{ firiva@sissa.it.}\\
	\hspace*{0.62cm}{Member of the Gruppo Nazionale per l'Analisi Matematica, la Probabilit\`a e le loro Applicazioni (GNAMPA)\\\hspace*{0.54 cm} of the Istituto Nazionale di Alta Matematica (INdAM).}}\\
	{\hspace*{0.42cm}$^2$ GLI Europe, \textsc{Diakenhuisweg, 29-35, 2033AP, Haarlem, The Netherlands.}\\ \hspace*{0.62cm}\textit{E-mail address}: \textsf{l.nardini@gaminglabs.com}}}
	
	\begin{abstract}
		In this paper we analyse a one--dimensional debonding model when viscosity is taken into account. It is described by the weakly damped wave equation whose domain, the debonded region, grows according to a Griffith's criterion. Firstly we prove that the equation admits a unique solution when the evolution of the debonding front is assigned. Finally we provide an existence and uniqueness result for the coupled problem given by the wave equation together with Griffith's criterion.	
	\end{abstract}

	\maketitle
		
	{\small
		\keywords{\noindent {\bf Keywords:}
			Dynamic debonding; Wave equation in time-dependent domains; Duhamel's principle; Dynamic energy release rate; Energy-dissipation balance; Maximum dissipation principle; Griffith's criterion.
		}
		\par
		\subjclass{\noindent {\bf 2010 MSC:}
			35L05, 
			35Q74, 
			35R35, 
			74H20, 
			74H25, 
			74H30. 
		}
	}

	\pagenumbering{arabic}

\tableofcontents

	\section*{Introduction}
	Analytical models of dynamic debonding involving a single spatial dimension have been developed in the last fifty years as a simplified but still meaningful version of dynamic crack growth based on Griffith's criterion. Starting from the works of Hellan~\cite{Hela, Helb} and Burridge~\&~Keller~\cite{BurrKell} (see also the books of Freund~\cite{Fre90} and Hellan~\cite{Helbook}) it is highlighted how in this field they are one of the few models for which a mathematical formulation provides an exhaustive description of the involved physical processes. Nevertheless they still possess all the relevant features and difficulties of general Fracture Dynamics, such as the time dependence of the domain of the wave equation and the presence of an energy criterion governing the evolution of the system. For the reader who is interested in recent works about dynamic crack propagation we quote for instance \cite{DMLar, DMLarToad, DMLuc, LarSla, RosTho}.\par 
	In the context of one-dimensional models a natural question of great interest in the framework of Fracture Mechanics, widely open in the general case, can be considered in detail. It is commonly referred as the quasistatic limit problem and it concerns whether or not dynamic solutions converge to a quasistatic evolution as inertia tends to zero. We refer to~\cite{Mielke15} for the abstract theory of quasistatic or, more precisely, rate-independent systems.\par	
	In recent years the model of a tape peeled away from a substrate has been studied from different points of view by several authors, see for instance \cite{DMLazNar16, DouMarCha08, LBDM12, LazNarkappa, LazNar,  LazNarinitiation}. In particular, a complete mathematical analysis has been given in~\cite{DMLazNar16, LazNar}, where the authors firstly prove well-posedness of the problem and then show how the quasistatic limit question has a negative answer in the undamped case.\par 
	In this work we contribute to the study of the same debonding model providing existence and uniqueness of dynamic evolutions when a viscous damping is taken into account. The issue of the quasistatic limit will be instead investigated in a future work. The choice of analysing such a damped problem is motivated by several works on different dynamic evolutions where the addition of a suitable dissipation term in the equation makes the convergence towards quasistatic solutions true, see for instance \cite{DMSca} in a case of perfect plasticity, \cite{Sca} for a model of delamination, \cite{LazRoss, Roub13} for some damage models, or~\cite{Ago, Nar15} in a finite--dimensional setting.\par 
	The model we consider can be interpreted in two different ways. The first one, following \cite{DouMarCha08, LBDM12}, describes a dynamic peeling test for a one--dimensional tape, which is assumed to be perfectly flexible and inextensible, initially attached to a flat rigid substrate and placed in some environment which causes a viscous damping on its surface. We assume the deformation of the tape takes place in a vertical plane with orthogonal coordinates $(x,y)$, where the positive $x$-axis represents the substrate as well as the reference configuration of the tape. For the sake of simplicity we neglect incompenetration between the tape and the substrate. During the evolution the tape is described by $x\mapsto(x+h(t,x),u(t,x))$, namely the pair $(h(t,x),u(t,x))$ is the displacement at time $t\ge 0$ of the point $(x,0)$, and it is glued to the substrate on the half line $\{x\ge \ell(t), y=0\}$, where $\ell$ is a nondecreasing function satisfying $\ell_0:=\ell(0)>0$ which represents the debonding front (this implies $h(t,x)=u(t,x)=0$ for $x\ge\ell(t)$). At the endpoint $x=0$ we prescribe a boundary condition $u(t,0)=w(t)$, where $w$ is the time-dependent vertical loading.\par 
	Linear approximation and inextensibility of the tape lead to the following formula for the horizontal displacement $h$:
	\begin{equation*}
		h(t,x) = \frac{1}{2} \int_x^{+\infty} u_x^2(t,\xi)\d \xi;
	\end{equation*}
	furthermore, introducing a parameter $\nu\ge 0$ which tunes viscosity, it turns out that the vertical displacement $u$ solves the problem
	\begin{equation}
\label{problem1}
\begin{cases}
u_{tt}(t,x)-u_{xx}(t,x)+\nu u_t(t,x)=0, \quad& t > 0 \,,\, 0<x<\ell(t),  \\
u(t,0)=w(t), &t>0, \\
u(t,\ell(t))=0,& t>0,\\
u(0,x)=u_0(x),\quad&0<x<\ell_0,\\
u_t(0,x)=u_1(x),&0<x<\ell_0,
\end{cases}
\end{equation}
	where the initial conditions $u_0$ and $u_1$ are given functions.\par
	The second and, in our opinion, much proper and simpler interpretation of the model is the one of a bar, initially glued to a flat rigid support, loaded horizontally and thus exhibiting only horizontal displacement. In this setting the function $u(t,x)$ represents the horizontal displacement of the bar, while $w(t)$ is the horizontal loading acting in $x=0$; as before, the nondecreasing function $\ell(t)$ denotes the debonding front, and system \eqref{problem1} governs the evolution of $u$.\par
	The addition of the damping term to the wave equation, harmless at a first sight, makes instead the problem much more difficult to treat than the undamped case $\nu=0$ previously analysed in \cite{DMLazNar16}. Indeed, the arguments they adopted do not work anymore because of a real coupling between the two unknowns $u$ and $\ell$ which appears if $\nu$ is positive. The aim of our contribution is thus to develop an original approach which allows us to overcome the technical difficulties related to the damping term and to get and improve the results obtained in \cite{DMLazNar16}.\par
	The paper is organised as follows: in Section~\ref{sec1} we prove there exists a unique solution $u$ to problem \eqref{problem1} when the evolution of the debonding front $\ell$ is known a priori; the idea is to introduce an equivalent problem solved by the function $v(t,x)=e^{\nu t/2}u(t,x)$ (see \eqref{problem2}) and then, exploiting a suitable representation formula (Duhamel's principle), to perform a contraction argument (see Proposition~\ref{contractionCl} and Theorem~\ref{globalsolutionprescr}).\par
	In Section~\ref{sec2} we study the total energy $\mc T$ of the solution $u$ to problem \eqref{problem1}, namely the sum of the internal energy and the energy dissipated by viscosity. We prove that $\mc T$ is an absolutely continuous function and we provide an explicit formula (for small time) for its derivative (see Proposition~\ref{energyderivative}).\par
	In the rest of the paper we take care of problem \eqref{problem1} when the evolution of the debonding front $\ell$ is unknown, but is governed by a suitable energy criterion (Griffith's criterion) based on the notion of dynamic energy release rate (see~\cite{Fre90} for its definition in the general framework of Fracture Mechanics); physically it represents the amount of energy for unit lenght spent to debond the tape.\par
	In the first part of Section~\ref{sec3} we introduce the dynamic energy release rate $G_\alpha(t)$ at time $t$ corresponding to a speed $\alpha\in(0,1)$ of the debonding front (see Definition \ref{galfa}) following the presentation given in~\cite{DMLazNar16}; in the second one we formulate Griffith's criterion (see \eqref{Griffithcrit}) under the assumption that the energy dissipated during the debonding process in the time interval $[0,t]$ is expressed by the formula
	\begin{equation*}
		\int_{\ell_0}^{\ell(t)}\kappa(x)\d x,
	\end{equation*}
	where $\kappa\colon[\ell_0,+\infty)\to(0,+\infty)$ is the local toughness of the glue between the tape and the substrate (for a more general case of speed-dependent toughness in the undamped case $\nu=0$ we refer to~\cite{LazNarkappa}). With this aim, as in~\cite{DMLazNar16} and~\cite{Lar10}, we formulate the evolution in terms of an energy-dissipation balance and of a maximum dissipation principle, deducing that $\ell$ must satisfy the following system, namely Griffith's criterion: 
	\begin{equation}\label{Griff}
	\begin{cases}
	\,\,0\le\dot{\ell}(t)<1,\\ 
	\,\,G_{\dot\ell(t)}(t)\le \kappa(\ell(t)),\\ 
	\left[ G_{\dot\ell(t)}(t)-\kappa(\ell(t))\right]\dot{\ell}(t)=0. 
	\end{cases}
	\end{equation}
	
	In Section~\ref{sec4} we present our main result: we solve the coupled problem showing existence and uniqueness of a pair $(u,\ell)$ satisfying \eqref{problem1}\&\eqref{Griff} (see Theorem~\ref{finalthm}). Our result generalises Theorem~3.5 in~\cite{DMLazNar16} both for the presence of the damping term as well as for the weaker regularity we require on the data. The strategy for the proof is, like in Section~\ref{sec1}, to rewrite \eqref{problem1}\&\eqref{Griff} as a fixed point problem and then to use a contraction argument (see Proposition~\ref{contractioncoupled}). Furthermore, our approach even allows us to consider the presence of an external force $f$ in the model (see Remark~\ref{external}), namely when the equation for the vertical displacement $u$ becomes
	\begin{equation*}
		u_{tt}(t,x)-u_{xx}(t,x)+\nu u_t(t,x)=f(t,x), \quad t > 0 \,,\, 0<x<\ell(t).
	\end{equation*}

	At the end of the work we attach an Appendix, in which we collect some results used through the paper about the Chain rule and the Leibniz differentiation rule under low regularity assumptions.
	
	\section{Prescribed debonding front}\label{sec1}
	In this Section we show existence and uniqueness of solutions to problem \eqref{problem1} when the evolution of the debonding front is prescribed. We first consider an auxiliary and equivalent problem, see \eqref{problem2}, which will be easier to handle than the original one; then we provide a representation formula, given by \eqref{Duhamel}, for a solution of this new problem. The result of existence and uniqueness will be finally obtained by means of a fixed point argument, as stated in Proposition~\ref{contractionCl} and Theorem~\ref{globalsolutionprescr}. We follow the same presentation given in~\cite{DMLazNar16}: we fix $\nu\ge 0$, $\ell_0>0$ and a function $\ell \colon [0,+\infty) \to [\ell_0,+\infty)$ such that
	\begin{subequations}\label{elle}
		\begin{equation}\label{ellea}
			\ell\in C^{0,1}([0,+\infty)) ,
		\end{equation}
		\begin{equation}\label{elleb}
			\ell(0)=\ell_0\mbox{ and } 0\le\dot\ell(t)\le 1 \mbox{ for a.e. }t\in [0,+\infty).
		\end{equation}
	\end{subequations}
	\begin{rmk}[\textbf{Notation}]
		Given any function of one variable $\phi\colon\erre\to\erre$ we always denote its derivative (when it exists) by $\dot \phi$, regardless of whether it is a time or a spatial derivative.
	\end{rmk}\noindent
	Differently from~\cite{DMLazNar16} we allow the debonding front $\ell$ to move even with speed one. For $t\in[0,+\infty)$ we introduce the functions:
	\begin{equation*}
		\varphi(t):= t {-} \ell(t) \,\mbox{ and } \psi(t):=t{+}\ell(t).
	\end{equation*} 
	Since $\psi$ is strictly increasing we can define 
	\begin{equation*}
		 \omega\colon [\ell_0,+\infty) \to [-\ell_0,+\infty) , \quad
		\omega(t):=\varphi\circ\psi^{-1}(t) ,
	\end{equation*}
	and we notice that $\omega$ is a Lipschitz function whose derivative satisfies for a.e. $t\in[\ell_0,+\infty)$ 
	\begin{equation}\label{deromega}
		0\le\dot\omega(t)=\frac{1-\dot\ell(\psi^{-1}(t))}{1+\dot\ell(\psi^{-1}(t))}\le 1.
	\end{equation}
	For $a\in\erre$ and for $k\ge0$ integer we introduce the spaces:
	\begin{gather*}
		\widetilde H^1(a,+\infty) :=\{u \in \Hloc(a,+\infty) \mid u\in H^1(a,b) \, \text{ for every } b>a\},\\
		\widetilde C^{k,1} ([a,+\infty)) := \{ u \in C^k([a,+\infty)) \mid u \in C^{k,1}([a,b]) \text{ for every } b>a \} .
	\end{gather*}
	We assume that
	\begin{subequations}\label{bdryregularity}
	\begin{equation}	
	w \in \widetilde  H^1(0,+\infty),	
	\end{equation}
	\begin{equation}\label{bdry2}
		u_0 \in H^1(0,\ell_0) , \quad u_1 \in L^2(0,\ell_0) .
	\end{equation}
	\end{subequations}	
	\begin{rmk}
		Throughout the paper every function in $W^{1,p}(a,b)$, for $-\infty<a<b<+\infty$ and $p\in[1,+\infty]$, is always identified with its continuous representative on $[a,b]$.
	\end{rmk}\noindent	
	For the initial data we require the compatibility conditions		
	\begin{equation}
			\label{compatibility0}
			u_0(0)=w(0), \quad u_0(\ell_0)=0.
	\end{equation}	
	We set:
	\begin{gather*}
		\Omega := \{ (t,x)\mid t>0\,,\, 0 < x < \ell(t)\},\\
		\Omega_T :=\{ (t,x)\in\Omega\mid t<T\}.
	\end{gather*}
	We will look for solutions in the space
	\begin{equation*}
		\widetilde{H}{^1}(\Omega) := \{u \in \Hloc(\Omega)\mid u \in H^1(\Omega_T) \mbox{ for every } T>0\},
	\end{equation*}
	or, assuming more regular data, in the space
	\begin{equation*}
		\widetilde{C}^{k,1}(\overline{\Omega}) := \{u \in C^k(\overline\Omega)\mid u \in C^{k,1}(\overline{\Omega_T}) \mbox{ for every } T>0\}.
	\end{equation*}
	\begin{defi}
	\label{sol}
	We say that a function $u \in \widetilde{H}{^1}(\Omega)$ (resp.\ in $H^1(\Omega_T)$) is a solution of \eqref{problem1} if $u_{tt}-u_{xx}+\nu u_{t}=0$  holds in the sense of distributions in $\Omega$ (resp.\ in $\Omega_T$), the boundary conditions are intended in the sense of traces and the initial conditions $u_0$ and $u_1$ are satisfied in the sense of $L^2(0,\ell_0)$ and $H^{-1}(0,\ell_0)$, respectively.
	\end{defi}
	\begin{rmk}
		The definition is well posed, since for a solution $u\in H^1(\Omega_T)$ we have that $u_t$ and $u_x$  belong to $L^2(0,T;L^2(0,\ell_0))$; this implies that $u_t$ and $u_{xx}$ are in $L^2(0,T;H^{-1}(0,\ell_0))$ and so by the wave equation $u_{tt}\in L^2(0,T;H^{-1}(0,\ell_0))$. Therefore $u_t\in H^1(0,T;H^{-1}(0,\ell_0))\subseteq C^0([0,T];H^{-1}(0,\ell_0))$ (see also~\cite{DMLazNar16}).
	\end{rmk}\noindent
One of the standard ways used to deal with the weakly damped wave equation consists in the introduction of the function $v(t,x):=e^{\nu t/2}u(t,x)$ (see for instance \cite{DauLio}, Remark~10, pag.~141), which in our setting solves the auxiliary problem
\begin{equation}
\label{problem2}
\begin{cases}
v_{tt}(t,x)-v_{xx}(t,x)-\displaystyle\frac{\nu^2}{4}v(t,x)=0, \quad& t > 0 \,,\, 0<x<\ell(t),  \\
v(t,0)=z(t), &t>0, \\
v(t,\ell(t))=0,& t>0,\\
v(0,x)=v_0(x),\quad&0<x<\ell_0,\\
v_t(0,x)=v_1(x),&0<x<\ell_0,
\end{cases}
\end{equation}	
where the boundary condition and the initial data are replaced respectively by the functions
\begin{equation}\label{datav}
	\begin{gathered}
	z(t)=e^{\nu t/2}w(t),\\
	v_0(x)=u_0(x)\quad\text{and}\quad v_1(x)=u_1(x)+\frac{\nu}{2}u_0(x).
	\end{gathered}
\end{equation}	
We notice that $z$, $v_0$ and $v_1$ in \eqref{datav} satisfy \eqref{bdryregularity} and the compatibility conditions \eqref{compatibility0} if and only if $w$, $u_0$ and $u_1$ do the same.
	\begin{rmk}\label{equivalent}
	It is easy to see that $u \in \widetilde{H}{^1}(\Omega)$ (resp. $H^1(\Omega_T)$) is a solution of \eqref{problem1} if and only if the corresponding function  $v(t,x)=e^{\nu t/2}u(t,x)\in \widetilde{H}{^1}(\Omega)$ (resp. $H^1(\Omega_T)$) is a solution of~\eqref{problem2}, according to Definition \ref{sol} (with the obvious changes). The absence of first derivatives in the equation for $v$ makes this second problem more convenient to deal with.	
	\end{rmk}\noindent
	We introduce also the sets (see Figure~\ref{Figrettangoli}):
	\begin{align*}	
	&\Omega'_1 \,:=\{ (t,x)\in\Omega\mid t\le x\mbox{ and }t+x\le \ell_0\},\\		
	&\Omega'_2 \,:=\{ (t,x)\in\Omega\mid t>x\mbox{ and }t+x<\ell_0\},\\	
	&\Omega'_3 \,:=\{ (t,x)\in\Omega\mid t<x\mbox{ and }t+x>\ell_0\},\\
	&\Omega'\,\, :=	\,\,\Omega'_1\cup\Omega'_2\cup\Omega'_3	,\\
	&\Omega'_T:=\{ (t,x)\in\Omega'\mid t<T\},
	\end{align*}
	and we consider the spaces:
	\begin{align*}
		&\widetilde{H}{^1}(\Omega') := \{u \in \Hloc(\Omega')\mid u \in H^1(\Omega'_T) \mbox{ for every } T>0\},\\
		&\widetilde{L}{^2}(\Omega') := \{u \in L^2_{\textnormal{loc}}(\Omega')\mid u \in L^2(\Omega'_T) \mbox{ for every } T>0\}.
	\end{align*}	
	In~\cite{DMLazNar16} it has been shown that every solution to the undamped (i.e. $\nu=0$) wave equation, here and henceforth denoted by $A(t,x)$, satisfies a suitable version of the classical d'Alembert's formula, adapted to the time dependence of the domain; imposing initial data and boundary conditions the authors prove that in $\Omega'$ it can be written as $A(t,x)=a_1(t{+}x)+a_2(t{-}x)$, where
	\begin{equation}\label{a1a2}
	\begin{aligned}
	&a_1(s)=\begin{cases}
	\displaystyle\frac 12 v_0(s)+\frac 12 \int_{0}^{s}v_1(r)\d r, &\mbox{ if }s\in(0,\ell_0],\\
	\displaystyle-\frac 12 v_0({-}\omega (s))+\frac 12 \int_{0}^{{-}\omega (s)}v_1(r)\d r, &\mbox{ if }s\in(\ell_0,2t^*),
	\end{cases}\\
	&a_2(s)=\begin{cases}
	\displaystyle\frac 12 v_0(-s)-\frac 12 \int_{0}^{-s}v_1(r)\d r, &\mbox{ if }s\in(-\ell_0,0],\\
	\displaystyle z(s)-\frac 12 v_0(s)-\frac 12 \int_{0}^{s}v_1(r)\d r, &\mbox{ if }s\in(0,\ell_0),
	\end{cases}
	\end{aligned}		
	\end{equation}
	with $t^*=\inf\{t\in[\ell_0,+\infty)\mid t=\ell(t)\}$ (with the convention $\inf\{\emptyset\}=+\infty$). We notice that by~\eqref{bdryregularity}, \eqref{compatibility0} and Remark~\ref{sobolev}, $a_1$ and $a_2$ belong to $\widetilde{H}^1(0,2t^*)$ and $H^1(-\ell_0,\ell_0)$ respectively; this will be used in Lemma~\ref{A}.
	\begin{rmk}
		We wrote $\widetilde{H}^1(0,2t^*)$ since $t^*$ can be $+\infty$; if this does not occur, that expression simply stands for $H^1(0,2t^*)$.
	\end{rmk}\noindent
	Hence d'Alembert's formula provides an explicit expression of $A$ in $\Omega'$:
	\begin{equation}\label{homsol}
	A(t,x)=\begin{cases}\displaystyle
	\frac 12 v_0(x{-}t)+\frac 12 v_0(x{+}t)+\frac 12 \int_{x{-}t}^{x{+}t}v_1(s) \d s , & \mbox{if } (t,x)\in\Omega_1',\\
	\displaystyle z(t{-}x)-\frac 12 v_0(t{-}x)+\frac 12 v_0(t{+}x)+\frac 12 \int_{t{-}x}^{t{+}x}v_1(s) \d s, & \mbox{if } (t,x)\in\Omega_2',\\
	\displaystyle\frac 12 v_0(x{-}t)-\frac 12 v_0({-}\omega(x{+}t))+\frac 12 \int_{x-t}^{{-}\omega(x{+}t)}v_1(s) \d s, & \mbox{if } (t,x)\in\Omega_3'.
	\end{cases}
	\end{equation}
	\begin{rmk}
		In $\Omega\setminus\Omega'$ one cannot anymore obtain explicit formulas for $a_1$, $a_2$, and hence for $A$, due to superpositions of forward and backward waves generated by ``bouncing'' against the endpoints $x=0$ and $x=\ell(t)$, even though d'Alembert's formula still holds true.
	\end{rmk}\noindent
	Inspired by the validity of this version of d'Alembert's formula in the undamped and homogeneous case $\nu=0$, to solve problem \eqref{problem2} we firstly prove that even the nonhomogeneous classical counterpart, the so called Duhamel's principle, holds true in our time-dependent domain setting. Duhamel's principle states that every solution to problem \eqref{problem2} can be written (in $\Omega'$) as a sum of two terms: the first one is the solution $A$ of the undamped wave equation, while the second one is the integral of the forcing term $\frac{\nu^2}{4}v(t,x)$ over a suitable space-time domain, denoted by $R(t,x)$. The domain of integration has the following form (see Figure~\ref{Figrettangoli}):
	\begin{equation}\label{rettangoli}
		R(t,x)= \{(\tau,\sigma) \in \Omega' \mid 0 < \tau < t,\,\,\, \gamma_1(\tau;t,x) < \sigma < \gamma_2(\tau;t,x) \},
	\end{equation}
	where
		\begin{equation}\label{bordi}
		\begin{aligned}		
		&\gamma_1(\tau;t,x) = \begin{cases}
		x{-}t{+}\tau, & \mbox{if } (t,x)\in\Omega_1', \\
		|x{-}t{+}\tau|, & \mbox{if } (t,x)\in\Omega_2',\\
		x{-}t{+}\tau, & \mbox{if } (t,x)\in\Omega_3',
		\end{cases}\,\\
		&\gamma_2(\tau;t,x) = \begin{cases}
		x{+}t{-}\tau, & \mbox{if } (t,x)\in\Omega_1',\\
		x{+}t{-}\tau, &\mbox{if } (t,x)\in\Omega_2',\\
		\tau{-}\omega(t{+}x), & \mbox{if } (t,x)\in\Omega_3' \mbox{ and } \tau \le \psi^{-1}(t{+}x),\\
		x{+}t{-}\tau, &  \mbox{if } (t,x)\in\Omega_3' \mbox{ and } \tau > \psi^{-1}(t{+}x),
		\end{cases}
		\end{aligned}
		\end{equation}		
	are the left and the right boundary of $R(t,x)$, respectively.\par
	\begin{figure}
			\subfloat{\includegraphics[scale=.4]{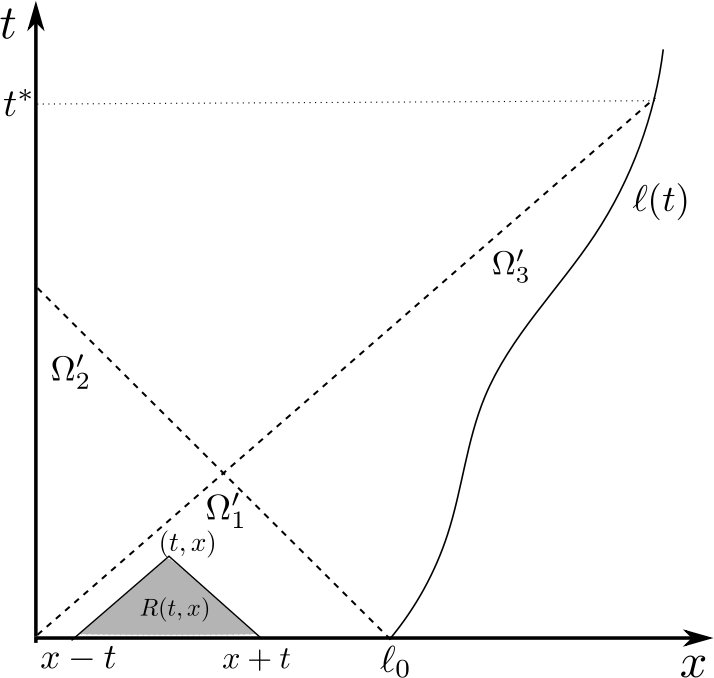}}\quad
			\subfloat{\includegraphics[scale=.4]{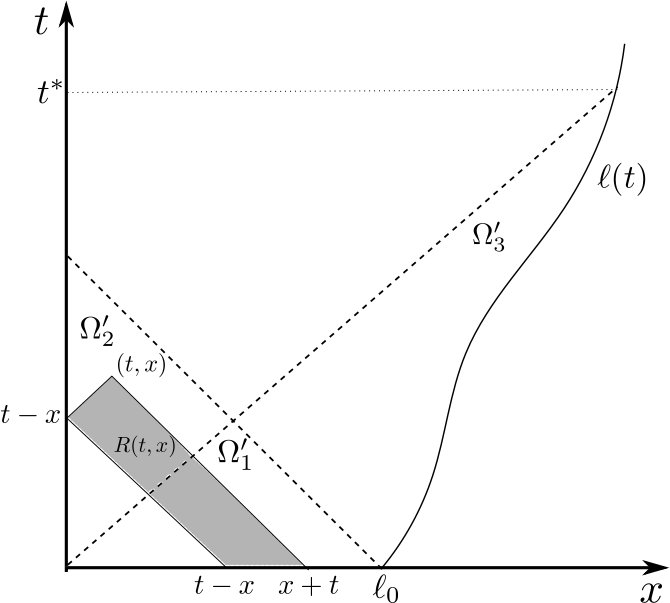}}\quad
			\subfloat{\includegraphics[scale=.4]{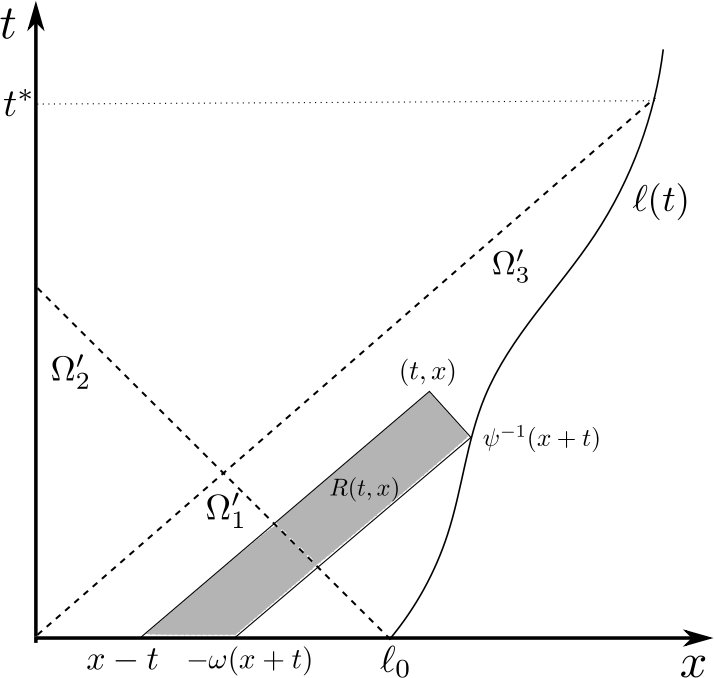}}
			\caption{The set $R(t,x)$ in the three possible cases $(t,x)\in\Omega'_1$, $(t,x)\in\Omega'_2$ and $(t,x)\in\Omega'_3$}.\label{Figrettangoli}
	\end{figure}
	The precise statement is the following:	
	\begin{prop}\label{representation}
		A function $v\in \widetilde{H}^1(\Omega')$ is a solution of \eqref{problem2} in $\Omega'$ if and only if
		\begin{equation}\label{Duhamel}
			v(t,x)=A(t,x)+\frac{\nu^2}{8}\iint_{R(t,x)}v(\tau,\sigma)\d\sigma \d\tau, \quad \mbox{for a.e. }(t,x)\in\Omega',
		\end{equation}
	where $A$ is as in \eqref{homsol} and $R$ is as in \eqref{rettangoli}.
	\end{prop}
\begin{proof}
	Let $v\in \widetilde{H}^1(\Omega')$ be a solution of \eqref{problem2} in $\Omega'$ and consider the change of variables
	\begin{equation}\label{changeofvariable}
			\begin{cases}
		\xi=t-x,\\
		\eta=t+x.
		\end{cases}
	\end{equation}
	Then the function $V(\xi,\eta):=v(\frac{\xi+\eta}{2},\frac{\eta-\xi}{2})$ satisfies (in the sense of distributions)
	\begin{equation}\label{eqchangevar}
		V_{\xi\eta}=\frac{\nu^2}{4} V \quad\quad\mbox{in  } \Lambda',
	\end{equation}
	where $\Lambda'$ is the image of $\Omega'$ through \eqref{changeofvariable}.\par
	Integrating \eqref{eqchangevar} over the image of $R(t,x)$ through \eqref{changeofvariable} and reverting to the original variables $(t,x)$ one gets representation formula \eqref{Duhamel} (imposing initial data and boundary conditions).\par
	Now assume that $v\in \widetilde{H}^1(\Omega')$ satisfies \eqref{Duhamel}; then using Lemma~\ref{derivatives} and recalling that $A_{tt}=A_{xx}$ (weakly) we can conclude.
\end{proof}
\begin{rmk}
	An analogous statement holds true for a solution $u$ of \eqref{problem1}, replacing \eqref{Duhamel} by
	\begin{equation}\label{Duhamelu}
	u(t,x)=\widehat{A}(t,x)-\frac{\nu}{2}\iint_{R(t,x)}u_t(\tau,\sigma)\d\sigma \d\tau, \quad \mbox{for a.e. }(t,x)\in\Omega',
	\end{equation}
	where $\widehat{A}$ is obtained replacing $v_0$, $v_1$ and $z$ by $u_0$, $u_1$ and $w$ in \eqref{homsol}.
\end{rmk}\noindent
	For a better understanding of the function $A$ and of the integral term we state the following two Lemmas.
	\begin{lemma}\label{A}
		 Fix $\ell_0>0$ and consider $v_0$, $v_1$ and $z$ satisfying \eqref{bdryregularity} and \eqref{compatibility0}. Assume that  $\ell\colon[0,+\infty)\to[\ell_0,+\infty)$ satisfies \eqref{elle}.\par 
		 Then the function $A$ defined in \eqref{homsol} is continuous on $\overline{\Omega'}$ and it belongs to $\widetilde{H}^1(\Omega')$; moreover, setting $A\equiv0$ outside $\overline{\Omega}$, for every $t\in\left[0,\frac{\ell_0}{2}\right]$ it holds true:
		 \begin{subequations}
		 	\begin{equation}\label{At}		 	
		 	\frac{A(t+h,\cdot)-A(t,\cdot)}{h}\xrightarrow[h\to 0]{}A_t(t,\cdot),\quad\text{a.e. in }(0,+\infty) \text{ and in }L^2(0,+\infty),
		 	\end{equation}
		 	\begin{equation}\label{Ax}
		 	\frac{A(t,\cdot+h)-A(t,\cdot)}{h}\xrightarrow[h\to 0]{}A_x(t,\cdot),\quad\text{a.e. in }(0,+\infty) \text{ and in }L^2(0,+\infty),
		 	\end{equation}
		\end{subequations}
		 where for every $t\in\left[0,\frac{\ell_0}{2}\right]$ and for a.e. $x\in(0,+\infty)$
		 \begin{equation*}
		 \begin{aligned}
		 &	A_t(t,x)=\begin{cases}
		 \dot{a}_1(t{+}x)+\dot{a}_2(t{-}x),&\text{if }x\in(0,\ell(t)),\\
		 0,&\text{if }x\in(\ell(t),+\infty),
		 \end{cases}\\
		&A_x(t,x)=\begin{cases}
		\dot{a}_1(t{+}x)-\dot{a}_2(t{-}x),&\text{if }x\in(0,\ell(t)),\\
		0,&\text{if }x\in(\ell(t),+\infty),
		\end{cases}		
		 \end{aligned}
	 \end{equation*}
		 being $a_1$ and $a_2$ as in \eqref{a1a2}.\par
		 Furthermore $A_t$ and $A_x$ belong to $C^0([0,\frac{\ell_0}{2}];L^2(0,+\infty))$ and hence in particular $A$ belongs to $C^0([0,\frac{\ell_0}{2}];H^1(0,+\infty))\cap C^1([0,\frac{\ell_0}{2}];L^2(0,+\infty))$.
	\end{lemma}
\begin{proof}
	By the following explicit expression of $A$,
	\begin{equation*}
		A(t,x)=\begin{cases}
		a_1(t{+}x)+a_2(t{-}x),&\text{for every }(t,x)\in\overline{\Omega'},\\
		0,&\text{for every }(t,x)\in[0,+\infty)^2\setminus\overline{\Omega},
		\end{cases}
	\end{equation*}
	and recalling that $a_1$ and $a_2$ belong to $\widetilde{H}^1(0,2t^*)$ and $H^1(-\ell_0,\ell_0)$ respectively, we deduce that $A\in\widetilde{H}^1(\Omega')\cap C^0(\overline{\Omega'})$.\par 
	By classical results on Sobolev functions and exploiting the fact that $A(t,\ell(t))=0$ for every $t\in[0,t^*]$ it is easy to see that for every $t\in\left[0,\frac{\ell_0}{2}\right]$ \eqref{Ax} holds. Similarly one can show that for every $t\in\left[0,\frac{\ell_0}{2}\right]$ the difference quotient in \eqref{At} converges to $A_t(t,x)$ for a.e. $x\in(0,+\infty)$; to prove that it converges even in the sense of $L^2(0,+\infty)$ we compute (we assume $h>0$, being the other case analogous):
	\begingroup
	\allowdisplaybreaks
	\begin{align*}
		&\quad\int_{0}^{+\infty}\left|\frac{A(t{+}h,x)-A(t,x)}{h}-A_t(t,x)\right|^2\d x\\
		&=\int_{0}^{\ell(t)}\left|\frac{a_1(t{+}h{+}x){-}a_1(t{+}x)}{h}-\dot{a}_1(t{+}x)+\frac{a_2(t{+}h{-}x){-}a_2(t{-}x)}{h}-\dot{a}_2(t{-}x)\right|^2\d x\\
		&\quad+\frac{1}{h^2}\int_{\ell(t)}^{\ell(t+h)}|A(t{+}h,x)|^2\d x.
	\end{align*}
	\endgroup
	The first integral tends to zero as $h\to0^+$ since $a_1$ and $a_2$ are Sobolev functions, while for the second one we argue as follows:
	\begingroup
	\allowdisplaybreaks
	\begin{align*}
		\frac{1}{h^2}\int_{\ell(t)}^{\ell(t+h)}\!\!\!\!\!\!|A(t{+}h,x)|^2\d x&=\frac{1}{h^2}\int_{\ell(t)}^{\ell(t+h)}\left|\int_{\ell(t+h)}^{x}\big(\dot{a}_1(t{+}h{+}s)-\dot{a}_2(t{+}h{-}s)\big)\d s\right|^2\d x\\
		&\le\frac{1}{h^2}\int_{\ell(t)}^{\ell(t+h)}\!\!\!\!\!\!\!(\ell(t{+}h)-\ell(t))\int_{\ell(t)}^{\ell(t+h)}\!\!\!\!\!\!\!\big|\dot{a}_1(t{+}h{+}s)-\dot{a}_2(t{+}h{-}s)\big|^2\d s\d x\\
		&=\left(\frac{\ell(t{+}h)-\ell(t)}{h}\right)^2\int_{\ell(t)}^{\ell(t+h)}\!\!\!\!\!\!\!\big|\dot{a}_1(t{+}h{+}s)-\dot{a}_2(t{+}h{-}s)\big|^2\d s\\
		&\le 2\int_{\ell(t)+t+h}^{\ell(t+h)+t+h}\!\!\!\!\!\!\!|\dot{a}_1(y)|^2\d y+2\int_{t+h-\ell(t+h)}^{t+h-\ell(t)}\!\!\!\!\!\!\!|\dot{a}_2(y)|^2\d y,
	\end{align*}
	\endgroup
	and by dominated convergence we deduce it goes to zero as $h\to0^+$ too, so \eqref{At} is proved.\par 
	The fact that $A_t$ and $A_x$ are continuous in $L^2(0,+\infty)$ follows from the continuity of translations in $L^2(0,+\infty)$, arguing as before.
\end{proof}
Next Lemma instead is related to the integral term appearing in \eqref{Duhamel}:
	\begin{lemma}\label{derivatives}
	Fix $\ell_0>0$ and assume that $\ell\colon[0,+\infty)\to[\ell_0,+\infty)$ satisfies \eqref{elle}. Let $F \in \widetilde L^2(\Omega')$ and for  every  $(t,x) \in \Omega'$ let 
	\begin{equation}\label{acca}
		H(t,x) =\iint_{R(t,x)}F(\tau,\sigma)\d\sigma \d\tau= \int_0^t \int_{\gamma_1(\tau;t,x)}^{\gamma_2(\tau;t,x)} F(\tau,\sigma) \d \sigma \d \tau.
	\end{equation} 
	
	Then $H$ is continuous on $\overline{\Omega'}$ and it belongs to $\widetilde{H}^1(\Omega')$; moreover, setting $H\equiv0$ outside $\overline{\Omega}$, for every $t\in\left[0,\frac{\ell_0}{2}\right]$ it holds true:
	\begin{subequations}
	\begin{equation}\label{Ht}		\frac{H(t+h,\cdot)-H(t,\cdot)}{h}\xrightarrow[h\to 0]{}H_t(t,\cdot),\quad\text{a.e. in }(0,+\infty) \text{ and in }L^2(0,+\infty),
	\end{equation}
	\begin{equation}\label{Hx}
	\frac{H(t,\cdot+h)-H(t,\cdot)}{h}\xrightarrow[h\to 0]{}H_x(t,\cdot),\quad\text{a.e. in }(0,+\infty) \text{ and in }L^2(0,+\infty),
	\end{equation}
	\end{subequations}
	where for every $t\in\left[0,\frac{\ell_0}{2}\right]$ and for a.e. $x\in(0,+\infty)$
	\begin{equation*}
	\begin{aligned}
	&H_t(t,x)=\begin{cases}
	\displaystyle\!\int_0^t \!\![F(\tau, \gamma_2(\tau;t,x))(\gamma_2)_t(\tau;t,x){-} F(\tau, \gamma_1(\tau;t,x))(\gamma_1)_t(\tau;t,x)]\d \tau,&\!\text{if }x\in(0,\ell(t)),\\
	0,&\!\text{if }x\in(\ell(t),+\infty).
	\end{cases}\\
	&H_x(t,x)=\begin{cases}
	\displaystyle\!\int_0^t \!\![F(\tau, \gamma_2(\tau;t,x))(\gamma_2)_x(\tau;t,x){-} F(\tau, \gamma_1(\tau;t,x))(\gamma_1)_x(\tau;t,x)]\d \tau,&\!\!\!\text{if }x\in(0,\ell(t)),\\
	0,&\!\!\!\text{if }x\in(\ell(t),+\infty).
	\end{cases}
	\end{aligned}
	\end{equation*}
	
	Furthermore $H_t$ and $H_x$ belong to $C^0([0,\frac{\ell_0}{2}];L^2(0,+\infty))$ and hence in particular $H$ belongs to $C^0([0,\frac{\ell_0}{2}];H^1(0,+\infty))\cap C^1([0,\frac{\ell_0}{2}];L^2(0,+\infty))$. 
\end{lemma}
\begin{proof}
	The continuity of $H$ in $\overline{\Omega'}$ follows from the absolute continuity of the integral.\par
	We define $\displaystyle G(\tau;t,x):=\int_{\gamma_1(\tau;t,x)}^{\gamma_2(\tau;t,x)} F(\tau,\sigma) \d \sigma$, so that $\displaystyle H(t,x)=\int_{0}^{t}G(\tau;t,x)\d\tau$, and we notice that for every $t\in\left[0,\frac{\ell_0}{2}\right]$ the function $(x,\tau)\mapsto G(\tau;t,x)$ satisfies the assumptions of Theorem~\ref{differentiation}; hence, exploiting the fact that $H(t,\ell(t))=0$ for every $t\in[0,t^*]$ and recalling Remark~\ref{remksobolev}, we get that $H(t,\cdot)$ belongs to $H^1(0,+\infty)$ and so \eqref{Hx} follows. By direct computations one can show that for every $t\in\left[0,\frac{\ell_0}{2}\right]$ the difference quotient in \eqref{Ht} converges to $H_t(t,x)$ for a.e. $x\in(0,+\infty)$; to prove that it converges even in the sense of $L^2(0,+\infty)$ we compute (we assume $h>0$):
	\begin{align*}
		\int_{0}^{+\infty}\left|\frac{H(t+h,x)-H(t,x)}{h}-H_t(t,x)\right|^2\d x=&\int_{0}^{\ell(t)}\left|\frac{H(t+h,x)-H(t,x)}{h}-H_t(t,x)\right|^2\d x\\
		&+\frac{1}{h^2}\int_{\ell(t)}^{\ell(t+h)}|H(t+h,x)|^2\d x.
	\end{align*}
	It is easy to see that the first integral goes to zero as $h\to0^+$, while for the second one we estimate:
	\begin{align*}
		\frac{1}{h^2}\int_{\ell(t)}^{\ell(t+h)}|H(t+h,x)|^2\d x&\le\frac{1}{h^2}\int_{\ell(t)}^{\ell(t+h)}\!\!\!\!|R(t+h,x)|\left(\iint_{R(t+h,x)}\!\!\!\!|F(\tau,\sigma)|^2\d\sigma\d\tau \right)\d x\\
		&\le \frac{1}{h^2}\int_{\ell(t)}^{\ell(t+h)}\!\!\!\!h(t+h)\left(\iint_{\widetilde{R}_h(t)}\!\!\!\!|F(\tau,\sigma)|^2\d\sigma\d\tau \right)\d x\\
		&\le(t+h)\left(\frac{\ell(t{+}h)-\ell(t)}{h}\right)\iint_{\widetilde{R}_h(t)}\!\!\!\!|F(\tau,\sigma)|^2\d\sigma\d\tau\\
		&\le(t+h)\iint_{\widetilde{R}_h(t)}\!\!\!\!|F(\tau,\sigma)|^2\d\sigma\d\tau=:(\ast),		
	\end{align*} 
	where we introduced the set $\widetilde{R}_h(t):=\{(\tau,\sigma)\in\Omega\mid0<\tau<t{+}h,\,\,\,\tau{-}t{-}h{+}\ell(t)<\sigma<\tau{-}t{+}\ell(t)\}$. By dominated convergence $(\ast)$ goes to zero as $h\to0^+$, so \eqref{Ht} is proved.\par 
	We conclude recalling that, arguing as before, the continuity of translations in $L^2(0,+\infty)$ ensures that $H_t$ and $H_x$ belong to $C^0([0,\frac{\ell_0}{2}];L^2(0,+\infty))$ (exploiting the definition of $\gamma_1$ and $\gamma_2$ given by \eqref{bordi}). In particular this yields $H\in\widetilde{H}^1(\Omega')$.
\end{proof}
\begin{rmk}\label{cunoellermk}
	By \eqref{bordi} one gets that for every $t\in\left[0,\frac{\ell_0}{2}\right]$ more explicit expressions for $H_t(t,\cdot)$ and $H_x(t,\cdot)$, valid for a.e. $x\in(0,\ell(t))$, are respectively
	\begin{subequations}\label{cunoelle}
			\begin{equation}\label{cunoelle2t}
		\!\!H_t(t,x)=\begin{cases}
			\displaystyle\int_{0}^{t}F(\tau,x{+}t{-}\tau)\d\tau+\displaystyle\int_{0}^{t}F(\tau,x{-}t{+}\tau)\d\tau,  &\Omega_1', \\
			\displaystyle\int_{0}^{t}F(\tau,x{+}t{-}\tau)\d\tau-	\displaystyle\int_{0}^{t{-}x}F(\tau,t{-}x{-}\tau)\d\tau+\displaystyle\int_{t{-}x}^{t}F(\tau,x{-}t{+}\tau)\d\tau, & \Omega_2',\\
			\displaystyle\int_{0}^{t}\!\!\!F(\tau,x{-}t{+}\tau)\d\tau{-}\dot\omega(x{+}t)\!\!\displaystyle\int_{0}^{\psi^{-1}(x{+}t)}\!\!\!\!\!\!\!\!\!\!\!\!\!\!\!\!\!\!\!\!\!F(\tau,\tau{-}\omega(x{+}t))\d\tau{+}\!\!\displaystyle\int_{\psi^{-1}(x{+}t)}^{t}\!\!\!\!\!\!\!\!\!\!\!\!\!\!\!\!\!\!\!\!F(\tau,x{+}t{-}\tau)\d\tau, &\Omega_3',
			\end{cases}
			\end{equation}
			\begin{equation}\label{cunoelle2x}
			\!\!H_x(t,x)=\begin{cases}
			\displaystyle\int_{0}^{t}F(\tau,x{+}t{-}\tau)\d\tau-\displaystyle\int_{0}^{t}F(\tau,x{-}t{+}\tau)\d\tau, & \Omega_1', \\
			\displaystyle\int_{0}^{t}F(\tau,x{+}t{-}\tau)\d\tau+	\displaystyle\int_{0}^{t{-}x}F(\tau,t{-}x{-}\tau)\d\tau-\displaystyle\int_{t{-}x}^{t}F(\tau,x{-}t{+}\tau)\d\tau, & \Omega_2',\\
			\!\!-\!\!\displaystyle\int_{0}^{t}\!\!\!\!\!F(\tau,x{-}t{+}\tau)\d\tau\!-\!\dot\omega(x{+}t)\displaystyle\int_{0}^{\psi^{-1}(x{+}t)}\!\!\!\!\!\!\!\!\!\!\!\!\!\!\!\!\!\!\!\!\!\!\!F(\tau,\tau{-}\omega(x{+}t))\d\tau\!+\!\!\displaystyle\int_{\psi^{-1}(x{+}t)}^{t}\!\!\!\!\!\!\!\!\!\!\!\!\!\!\!\!\!\!\!\!\!F(\tau,x{+}t{-}\tau)\d\tau, & \Omega_3'.
			\end{cases}
			\end{equation}
	\end{subequations}
\end{rmk}
Since by Lemmas~\ref{A} and \ref{derivatives} the right-hand side in \eqref{Duhamel} is continuous on $\overline{\Omega'}$, every solution $v\in\widetilde{H}^1(\Omega')$ of problem \eqref{problem2} admits a representative, still denoted by $v$, which is continuous on $\overline{\Omega'}$ and such that (exploiting \eqref{homsol} and \eqref{acca}):
\begin{itemize}
	\item[-] $v(t,\ell(t))=0$ for every $t\in[0,t^*]$,
	\item[-] $v(t,0)=z(t)$ for every $t\in\left[0,\ell_0\right]$,
	\item[-] $v(0,x)=v_0(x)$ for every $x\in[0,\ell_0]$.
\end{itemize}
Moreover (the continuous representative of) the solution $v$ belongs to $C^0([0,\frac{\ell_0}{2}];H^1(0,+\infty))$ and to $C^1([0,\frac{\ell_0}{2}];L^2(0,+\infty))$ and  by \eqref{At}, \eqref{Ht} and \eqref{homsol}, \eqref{cunoelle} we deduce:
\begin{itemize}
	\item[-] $v_t(t,\cdot)\xrightarrow[t\to 0^+]{L^2(0,\ell_0)}v_1$,
	\item[-] $v_t(0,x)=v_1(x)$ for a.e. $x\in[0,\ell_0]$.
\end{itemize} 

In order to find existence (and uniqueness) of solutions to problem \eqref{problem2}, and hence to problem~\eqref{problem1}, we look for a fixed point of the linear operator $\mathfrak L\colon C^0(\overline{\Omega'})\to C^0(\overline{\Omega'})$ defined as: \begin{equation}\label{Phi}
	\mathfrak L v(t,x):=A(t,x)+\frac{\nu^2}{8}\iint_{R(t,x)}v(\tau,\sigma)\d\sigma \d\tau .
\end{equation}
	\begin{prop}\label{contractionCl}
		Fix $\nu\ge0$, $\ell_0>0$ and consider $v_0$, $v_1$ and $z$ satisfying \eqref{bdryregularity} and \eqref{compatibility0}. Assume that $\ell\colon [0,+\infty)\to[\ell_0,+\infty)$ satisfies \eqref{elle}.\par 
		If $ T\in\left(0,\frac{\ell_0}{2}\right)$ satisfies $\nu^2\ell_0 T<4$, then the map $\mathfrak L$ in \eqref{Phi} is a contraction from $C^0(\overline{\Omega_T})$ into itself.
	\end{prop}
	\begin{proof}
		By Lemmas~\ref{A} and \ref{derivatives} operator $\mathfrak L$ maps $C^0(\overline{\Omega_T})$ into itself. Pick $v^1,\,v^2\in C^0(\overline{\Omega_T})$ and let $(t,x)\in\overline{\Omega_T}$, then
		\begin{align*}
			|\mathfrak Lv^1(t,x)-\mathfrak Lv^2(t,x)|&\le\frac{\nu^2}{8}\iint_{R(t,x)}|v^1(\tau,\sigma)-v^2(\tau,\sigma)|\d\sigma\d\tau\le\frac{\nu^2}{8}|R(t,x)|\Vert v^1-v^2\Vert_{C^0(\overline{\Omega_{T}})}\\
			&\le\frac{\nu^2}{8}|\Omega_T|\Vert v^1-v^2\Vert_{C^0(\overline{\Omega_{T}})}\le\frac{\nu^2\ell_0 T}{4}\Vert v^1-v^2\Vert_{C^0(\overline{\Omega_{T}})}.
		\end{align*}
		Since $\nu^2\ell_0T<4$ we conclude.
	\end{proof}
We are now in a position to state and prove the first main result of the paper, regarding the existence and uniqueness of solutions of \eqref{problem2}, and hence of \eqref{problem1} (see Remark~\ref{equivalent}), when the debonding front $\ell$ is assigned:
\begin{thm}\label{globalsolutionprescr}
	Fix $\nu\ge0$, $\ell_0>0$ and consider $v_0$, $v_1$ and $z$ satisfying \eqref{bdryregularity} and \eqref{compatibility0}. Assume that $\ell\colon [0,+\infty)\to[\ell_0,+\infty)$ satisfies \eqref{elle}.\par 
	Then there exists a unique $v\in\widetilde{H}{^1}(\Omega)$ solution of \eqref{problem2}. Moreover $v$ has a continuous representative on $\overline{\Omega}$, still denoted by $v$, and, setting $v\equiv0$ outside $\overline{\Omega}$, it holds:
	\begin{equation*}
		v\in C^0([0,+\infty);H^1(0,+\infty))\cap C^1([0,+\infty);L^2(0,+\infty)). 
	\end{equation*}
	\end{thm}
\begin{proof}
	By Proposition~\ref{contractionCl} we deduce the existence of a unique continuous function $v^1$ satisfying~\eqref{Duhamel} in $\overline{\Omega_{T_1}}$, taking for istance $\displaystyle T_1=\frac 12\min\left\{\frac{\ell_0}{2},\frac{4}{\nu^2\ell_0}\right\}$ $\left(\displaystyle T_1=\frac{\ell_0}{4}\text{ if } \nu=0\right)$.\par
	By Lemmas~\ref{A} and \ref{derivatives} one gets that $v^1$ is in ${H}{^1}(\Omega_{T_1})$ and moreover that it belongs to $  C^0([0,T_1];H^1(0,+\infty))\cap C^1([0,T_1];L^2(0,+\infty))$, while Proposition~\ref{representation} ensures that $v^1$ solves problem \eqref{problem2} in $\Omega_{T_1}$.\par
	Now we can restart the argument from time $T_1$ replacing $\ell_0$ by $\ell_1:=\ell(T_1)$, $v_0$ by $v^1(T_1,\cdot)$ and $v_1$ by $v^1_t(T_1,\cdot)$; indeed notice that $v^1(T_1,\cdot)\in H^1(0,\ell_1)$, $v^1_t(T_1,\cdot)\in L^2(0,\ell_1)$ and that they satisfy the compatibility conditions $v^1(T_1,0)=z(T_1)$ and $v^1(T_1,\ell_1)=0$. Arguing as before we get the existence of a unique solution $v^2$ of \eqref{problem2} in $\Omega_{T_2}\setminus\Omega_{T_1}$, with $\displaystyle T_2=T_1+\frac 12\min\left\{\frac{\ell_1}{2},\frac{4}{\nu^2\ell_1}\right\}$, belonging to $C^0([T_1,T_2];H^1(0,+\infty))\cap C^1([T_1,T_2];L^2(0,+\infty))$.\par
	Then the function $v(t,x)=\begin{cases}
	v^1(t,x),&\mbox{if }(t,x)\in\overline{\Omega_{T_1}},\\
	v^2(t,x),&\mbox{if }(t,x)\in\overline{\Omega_{T_2}\setminus\Omega_{T_1}},
	\end{cases}$ is in $C^0([0,T_2];H^1(0,+\infty))$ and in $ C^1([0,T_2];L^2(0,+\infty))$ and it is easy to see that it is the only solution of \eqref{problem2} in $\Omega_{T_2}$.\par 
	To conclude we need to prove that the sequence of times $\{T_k\}$ defined recursively by
	\begin{equation*}
		\begin{cases}
		\displaystyle T_k=T_{k-1}+\frac 12\min\left\{\frac{\ell(T_{k-1})}{2},\frac{4}{\nu^2\ell(T_{k-1})}\right\}, &\text{if }k\ge 1,\\
		T_0=0,
		\end{cases}
	\end{equation*}
	diverges. This follows easily observing that $\{T_k\}$ is increasing and recalling that $0<\ell(t)<+\infty$ for every $t\in[0,+\infty)$.
	\end{proof}
	\begin{rmk}[\textbf{Regularity}]\label{lipschitzregglob}
	If we assume $v_0\in C^{0,1}([0,\ell_0])$, $v_1\in L^\infty(0,\ell_0)$, $z\in\widetilde{C}{^{0,1}}([0,+\infty)) $ satisfy the compatibility conditions \eqref{compatibility0}, then by \eqref{homsol} and \eqref{cunoelle} the (continuous representative of the) solution $v$ belongs to $\widetilde C^{0,1}(\overline{\Omega})$ and $v_t(t,\cdot)\in L^\infty(0,+\infty)$ for every $t\in[0,+\infty)$.
	\end{rmk}
	\begin{rmk}[\textbf{More regularity}]\label{morereg}
		If we assume more regularity on $v_0$, $v_1$, $z$ and on the debonding front $\ell$, in order to get that the solution $v$ possesses the same regularity we need to add more compatibility conditions. For instance, if $\ell\in \widetilde C^{1,1}([0,+\infty))$ satisfies \eqref{elleb}, if $v_0\in C^{1,1}([0,\ell_0])$, $v_1\in C^{0,1}([0,\ell_0])$, $z\in\widetilde{C}{^{1,1}}([0,+\infty)) $ satisfy \eqref{compatibility0}, to get $v\in \widetilde C^{1,1}(\overline{\Omega})$ we also need to assume the following first order compatibility conditions:
		\begin{equation}\label{comp1}
			v_1(0)=\dot z(0)\quad\mbox{ and }\quad v_1(\ell_0)+\dot{\ell}(0)\dot v_0(\ell_0)=0.
		\end{equation}
		Indeed, under these assumptions the function $A$ in \eqref{homsol} belongs to $\widetilde C^{1,1}(\overline{\Omega'})$; moreover, exploiting~\eqref{cunoelle} and the fact that by Remark~\ref{lipschitzregglob} we already know that the solution $v$ is in $\widetilde C^{0,1}(\overline{\Omega})$, one can deduce that the function $H(t,x) =\displaystyle\iint_{R(t,x)}v(\tau,\sigma)\d\sigma \d\tau$ in \eqref{acca} belongs to $\widetilde C^{1,1}(\overline{\Omega'})$ too. Hence representation formula \eqref{Duhamel} ensures that $v$ belongs to $C^{1,1}(\overline{\Omega_{T_1}})$ for some $T_1\in\left(0,\frac{\ell_0}{2}\right)$; since $v(t,0)=z(t)$ and $v(t,\ell(t))=0$ for every $t\in[0,+\infty)$ we notice that condition \eqref{comp1} holds at time $T_1$ too, and reasoning as in the proof of Theorem~\ref{globalsolutionprescr} one can conclude.\par 
		We also notice that, coming back to $u_0$, $u_1$ and $w$, \eqref{comp1} is equivalent to
		\begin{equation}\label{comp2}
		u_1(0)=\dot w(0)\quad\mbox{ and }\quad u_1(\ell_0)+\dot{\ell}(0)\dot u_0(\ell_0)=0.
		\end{equation}
	\end{rmk}
We conclude this first Section pointing out that the choice of working with $H^1$ and $L^2$ functions is only due to the energetic considerations we make in the next Sections in order to formulate the coupled problem. Indeed all the results presented up to now still remains valid in a $W^{1,1}$ and $L^1$ setting, with the obvious changes. 
	\section{Energetic analysis}\label{sec2}
	This Section is devoted to the study of the total energy of the solution $u$ to problem \eqref{problem1} given by Theorem~\ref{globalsolutionprescr} and Remark~\ref{equivalent}; this analysis will be used in Section~\ref{sec3} to introduce the notion of dynamic energy release rate.\par
	Fix $\nu\ge 0$, $\ell_0>0$ and a function $\ell \colon [0,+\infty) \to [\ell_0,+\infty)$ satisfying \eqref{elle}, and consider $u_0$, $u_1$ and $w$ satisfying \eqref{bdryregularity} and \eqref{compatibility0}; let $u$ be the solution of \eqref{problem1} associated with $\ell$, $u_0$, $u_1$ and $w$. For $t\in[0,+\infty)$ we introduce the internal energy of $u$:
	\begin{equation*}
		\mathcal{E}(t):=\frac 12 \int_{0}^{\ell(t)}\left(u_t^2(t,x)+u_x^2(t,x)\right)\d x,
	\end{equation*}
	 where the first term represents the kinetic energy and the second one the potential energy, and the energy dissipated by viscosity:
	 \begin{equation*}
	 \mathcal{A}(t):=\nu\int_{0}^{t}\int_{0}^{\ell(\tau)}u_t^2(\tau,\sigma)\d\sigma\d\tau.
	 \end{equation*}
	We then consider the total energy of $u$:
	\begin{equation}\label{totalenergy}
		 \mathcal{T}(t):=\mathcal{E}(t)+\mathcal{A}(t).
	\end{equation}
	As in Section~\ref{sec1} we introduce the auxiliary function $v(t,x)=e^{\nu t/2}u(t,x)$ and we consider $v_0$ and $v_1$ given by \eqref{datav}.
	\begin{prop}\label{energyderivative}
		The total energy $\mathcal{T}$ defined in \eqref{totalenergy} belongs to $ AC([0,+\infty))$ and for a.e. $t\in\left[0,\frac{\ell_0}{2}\right]$ the following formulas hold true:
		\begin{subequations}
		\begin{align}		
		&\begin{aligned}\label{totderu}
			\dot{\mathcal{T}}(t)=&-\frac{\dot\ell(t)}{2}\frac{1-\dot\ell(t)}{1+\dot\ell(t)}\left[\dot u_0(\ell(t){-}t)-u_1(\ell(t){-}t)+\nu\int_{0}^{t}u_t(\tau,\tau{-}t{+}\ell(t))\d\tau\right]^2\\
			&+\dot w(t)\left[\dot w(t)-\left(\dot u_0(t)+u_1(t)-\nu\int_{0}^{t}u_t(\tau,t{-}\tau)\d\tau\right)\right],
		\end{aligned}\\
			&\begin{aligned}\label{totderv}
			\dot{\mathcal{T}}(t)=&-\frac{\dot\ell(t)}{2}\frac{1-\dot\ell(t)}{1+\dot\ell(t)}e^{-\nu t}\left[\dot v_0(\ell(t){-}t)-v_1(\ell(t){-}t)-\frac{\nu^2}{4}\int_{0}^{t}v(\tau,\tau{-}t{+}\ell(t))\d\tau\right]^2\\
			&+\dot w(t)\left[\dot w(t)+\frac{\nu}{2}w(t)-e^{-\frac{\nu t}{2}}\left(\dot v_0(t)+v_1(t)+\frac{\nu^2}{4}\int_{0}^{t}v(\tau,t{-}\tau)\d\tau\right)\right],
			\end{aligned}
		\end{align}	
	\end{subequations}
		where the products between $1-\dot\ell(t)$ and the expressions within square brackets are meant as in Remark~\ref{precisini}.
	\end{prop}
\begin{rmk}\label{alltimes}
	One can obtain similar formulas for $\dot{\mc T}$ which are valid for a.e. $t\in[0,+\infty)$ arguing in the following way: fix $t_0>0$, then for a.e. $t\in\left[t_0,t_0+\frac{\ell(t_0)}{2}\right]$
	\begin{align*}
	\dot{\mathcal{T}}(t)=&-\frac{\dot\ell(t)}{2}\frac{1-\dot\ell(t)}{1+\dot\ell(t)}\left[u_x(t_0,\ell(t){-}t{+}t_0)-u_t(t_0,\ell(t){-}t{+}t_0)+\nu\int_{t_0}^{t}u_t(\tau,\tau{-}t{+}\ell(t))\d\tau\right]^2\\
	&+\dot w(t)\left[\dot w(t)-\left(u_x(t_0,t{-}t_0)+u_t(t_0,t{-}t_0)-\nu\int_{t_0}^{t}u_t(\tau,t{-}\tau)\d\tau\right)\right].
	\end{align*}
	and the analogous formula for \eqref{totderv} holds.
\end{rmk}
\begin{proof}[Proof of Proposition~\ref{energyderivative}]
	Let us define $T:=\ell_0/2$; we notice that by Remark~\ref{alltimes} it is enough to prove the Proposition in the time interval $[0,T]$. By \eqref{Duhamelu} we know that for every $(t,x)\in\overline{\Omega_T}$
	\begin{equation}\label{serve}
		u(t,x)=\hat{a}_1(t{+}x)+\hat{a}_2(t{-}x)-\frac{\nu}{2}\iint_{R(t,x)}u_t(\tau,\sigma)\d\tau\d\sigma,
	\end{equation}
	where $\hat{a}_1$ and $\hat{a}_2$ are as in \eqref{a1a2}, replacing $v_0$, $v_1$ and $z$ by $u_0$, $u_1$ and $w$, respectively.\par
	Moreover, by \eqref{serve}, Lemma~\ref{derivatives} and Remark~\ref{cunoellermk} we get for every $t\in[0,T]$
	\begin{subequations}
		\begin{equation*}
			u_t(t,x)=\dot{\hat{a}}_1(t{+}x)+\dot{\hat{a}}_2(t{-}x)-\frac{\nu}{2} h_1(t,x)-\frac{\nu}{2} h_2(t,x),\quad\quad\mbox{ for a.e. }x\in[0,\ell(t)],
		\end{equation*}
			\begin{equation*}
			u_x(t,x)=\dot{\hat{a}}_1(t{+}x)-\dot{\hat{a}}_2(t{-}x)-\frac{\nu}{2} h_1(t,x)+\frac{\nu}{2} h_2(t,x),\quad\quad\mbox{ for a.e. }x\in[0,\ell(t)],
			\end{equation*}
	\end{subequations}
where
\begin{align*}
	&h_1(t,x)=\begin{cases}
	\displaystyle\int_{0}^{t}u_t(\tau,t{+}x{-}\tau)\d\tau,&\mbox{if }0\le x\le\ell_0{-}t,\\
	{-}\dot\omega(t{+}x)\displaystyle\int_{0}^{\psi^{-1}(t{+}x)}\!\!\!\!\!\!\!\!\!\!\!\!\!\!\!\!u_t(\tau,\tau{-}\omega(t{+}x))\d\tau{+}\displaystyle\int_{\psi^{-1}(t{+}x)}^{t}\!\!\!\!\!\!\!\!\!\!\!\!\!\!\!u_t(\tau,t{+}x{-}\tau)\d\tau,&\mbox{if }\ell_0{-}t< x\le\ell(t),
	\end{cases}\\
&h_2(t,x)=\begin{cases}
\displaystyle\int_{0}^{t}u_t(\tau,\tau{-}t{+}x)\d\tau,&\mbox{if }t\le x\le\ell(t),\\
-\displaystyle\int_{0}^{t{-}x}u_t(\tau,t{-}x{-}\tau)\d\tau+\displaystyle\int_{t{-}x}^{t}u_t(\tau,\tau{-}t{+}x)\d\tau,&\mbox{if }0\le x<t.
\end{cases}
\end{align*}
Now we compute:
\begingroup
\allowdisplaybreaks
\begin{align*}
	\mathcal{E}(t)&=\frac 12\int_{0}^{\ell(t)}\left(\dot{\hat{a}}_1(t{+}x)+\dot{\hat{a}}_2(t{-}x)-\frac{\nu}{2} h_1(t,x)-\frac{\nu}{2} h_2(t,x)\right)^2\d x\\
	&\quad+\frac 12\int_{0}^{\ell(t)}\left(\dot{\hat{a}}_1(t{+}x)-\dot{\hat{a}}_2(t{-}x)-\frac{\nu}{2} h_1(t,x)+\frac{\nu}{2} h_2(t,x)\right)^2\d x\\
	&=\frac 12\int_{0}^{\ell(t)}\left[\left(\dot{\hat{a}}_1(t{+}x)-\frac{\nu}{2} h_1(t,x)\right)^2+\left(\dot{\hat{a}}_2(t{-}x)-\frac{\nu}{2} h_2(t,x)\right)^2\right]\d x\\
	&=\int_{t}^{t{+}\ell(t)}\left[\dot{\hat{a}}_1(y)-\frac{\nu}{2} h_1(t,y{-}t)\right]^2\d y+\int_{t{-}\ell(t)}^{t}\left[\dot{\hat{a}}_2(y)-\frac{\nu}{2} h_2(t,t{-}y)\right]^2\d y\\
	&=\int_{t}^{\ell_0}\left[\frac{\dot{u}_0(y)+u_1(y)}{2}-\frac{\nu}{2}\int_{0}^{t}u_t(\tau,y{-}\tau)\d\tau\right]^2\d y\\
	&\quad +\int_{t{-}\ell(t)}^{0}\left[\frac{\dot{u}_0({-}y)-u_1({-}y)}{2}+\frac{\nu}{2}\int_{0}^{t}u_t(\tau,\tau{-}y)\d\tau\right]^2\d y\\
	&\quad+\int_{\ell_0}^{t{+}\ell(t)}\!\!\left[\dot\omega(y)\left(\frac{\dot{u}_0({-}\omega(y)){-}u_1({-}\omega(y))}{2}{+}\frac{\nu}{2}\int_{0}^{\psi^{-1}(y)}\!\!\!\!\!\!\!\!\!\!\!\!\!\!\!\!\!u_t(\tau,\tau{-}\omega(y))\d\tau\right)\!\!{-}\frac{\nu}{2}\int_{\psi^{-1}(y)}^{t}\!\!\!\!\!\!\!\!\!\!\!\!\!u_t(\tau,y{-}\tau)\d\tau\right]^2\!\!\!\!\!\d y\\
	&\quad+\int_{0}^{t}\left[\dot w(y)-\frac{\dot{u}_0(y)+u_1(y)}{2}+\frac{\nu}{2}\int_{0}^{y}u_t(\tau,y{-}\tau)\d\tau-\frac{\nu}{2}\int_{y}^{t}u_t(\tau,\tau{-}y)\d\tau\right]^2\d y.
\end{align*}
\endgroup
It is easy to check that we can apply Theorem~\ref{differentiation} in the Appendix, so we obtain that $\mc E$ belongs to $AC([0,T])$ and that for a.e. $t\in[0,T]$ the following formula for its derivative holds true: 
\begin{align*}
	\dot{\mathcal{E}}(t)=&-\frac{\dot\ell(t)}{2}\frac{1-\dot\ell(t)}{1+\dot\ell(t)}\left[\dot u_0(\ell(t){-}t)-u_1(\ell(t){-}t)+\nu\int_{0}^{t}u_t(\tau,\tau{-}t{+}\ell(t))\d\tau\right]^2\\
	&+\dot w(t)\left[\dot w(t)-\left(\dot u_0(t)+u_1(t)-\nu\int_{0}^{t}u_t(\tau,t{-}\tau)\d\tau\right)\right]-\nu\int_{0}^{\ell(t)}u_t^2(t,x)\d x.
\end{align*}
Recalling that $\mc A$ is absolutely continuous by construction and that $\dot{\mc A}(t)=\displaystyle\nu\int_{0}^{\ell(t)}u_t^2(t,x)\d x$ for a.e. $t\in[0,T]$, we deduce that $\mc T$ belongs to $AC([0,T])$ and that formula \eqref{totderu} holds.\par 
To get \eqref{totderv} one argues in the same way with $v(t,x)=e^{\nu t/2}u(t,x)$, rewriting $\mathcal{E}$ as 
\begin{equation*}
	\mathcal{E}(t)=\frac{e^{-\nu t}}{2}\int_{0}^{\ell(t)}\left[\left(v_t(t,x)-\frac{\nu}{2}v(t,x)\right)^2+v_x^2(t,x)\right]\d x,
\end{equation*}
and recalling \eqref{Duhamel}.
\end{proof}
\section{Principles leading the debonding growth}\label{sec3}
In the first part of this Section we introduce the dynamic energy release rate in the context of our model, following~\cite{DMLazNar16}. In the second one we will use it to formulate Griffith's criterion, namely the energy criterion which rules the evolution of the debonding front.\par 
As before we fix $\nu\ge0$, $\ell_0>0$ and we consider $u_0$, $u_1$ and $w$ satisfying \eqref{bdryregularity} and \eqref{compatibility0}, but from now on the debonding front will be a function $\ell \colon [0,+\infty) \to [\ell_0,+\infty)$ satisfying \eqref{ellea} and such that
	\begin{equation}\label{ellemin}
	\ell(0)=\ell_0\mbox{ and } 0\le\dot\ell(t)< 1 \mbox{ for a.e. }t\in [0,+\infty).
	\end{equation}
	We want to underline that the requirement of \eqref{ellemin} in place of \eqref{elleb} is not merely a technical assumption needed to carry out all the mathematical arguments of the next Sections, although is crucial; it is instead a natural consequence of the Griffith's criterion the debonding front has to fulfill during its evolution, as the reader can check from the final formula \eqref{equation}.
\subsection{Dynamic energy release rate}
The notion of dynamic energy release has been developed in the framework of Fracture Mechanics to measure the amount of energy spent by the growth of the crack (see~\cite{Fre90} for more information); it is defined as the opposite of the derivative of the energy with respect to the measure of the evolved crack.\par  
To define it in the context of our debonding model we argue as in~\cite{DMLazNar16}: we fix $\bar t>0$ and we consider a function  $\tilde{w}\in \widetilde H^1(0,+\infty)$ and a function $\tilde\ell\colon [0,+\infty) \to [\ell_0,+\infty)$ satisfying \eqref{ellea}~and~\eqref{ellemin}, and such that 
\begin{equation*}
	\tilde{w}(t)= w(t) \quad \text{and} \quad
	\tilde\ell(t)=\ell(t) \quad \text{for every}\ t \in[0,\bar t \,] . 
	\end{equation*}
Let $u$ and $\tilde u$ be the solutions to problem \eqref{problem1} corresponding to $\ell$, $u_0$, $u_1$, $w$ and $\tilde\ell$, $u_0$, $u_1$, $\tilde w$, respectively, and for $t\in[0,+\infty)$ let us consider:
\begin{equation*}
	\mathcal{E}(t;\tilde{\ell},\tilde{w}):=\frac 12 \int_{0}^{\tilde\ell(t)}\left(\tilde u_t^2(t,x)+\tilde u_x^2(t,x)\right)\d x,
\end{equation*} 
\begin{equation*}
\mathcal{A}(t;\tilde{\ell},\tilde{w}):=\nu\int_{0}^{t} \int_{0}^{\tilde\ell(\tau)}\tilde u_t^2(\tau,\sigma)\d \sigma\d \tau,
\end{equation*} 
and
\begin{equation*}
\mathcal{T}(t;\tilde{\ell},\tilde{w}):=\mathcal{E}(t;\tilde{\ell},\tilde{w})+\mathcal{A}(t;\tilde{\ell},\tilde{w}),
\end{equation*}
where we stressed the dependence on $\tilde{\ell}$ and on $\tilde{w}$.\par
The formal definition of dynamic energy release rate at time $\bar t$ should be:
\begin{equation}\label{formal}
G(\bar t\,):=\lim\limits_{t\to\bar t^+}-\frac{\mathcal{T}(t;\tilde{\ell},\bar{w})-\mathcal{T}(\bar t;\ell,w)}{\tilde{\ell}(t)-\ell(\bar t\,)}=-\frac{1}{\dot{\tilde{\ell}}(\bar t\,)}\lim\limits_{t\to\bar t^+}\frac{\mathcal{T}(t;\tilde{\ell},\bar{w})-\mathcal{T}(\bar t;\ell,w)}{t-\bar t},
\end{equation}
where $\bar w\in\widetilde H^1(0,+\infty) $ is the constant extension of $w$ after $\bar t$.
\begin{rmk}
	The choice of the particular extension $\bar w$ in \eqref{formal} is needed in order to avoid including the work done by the external loading in the energy dissipated to debond the tape. 
\end{rmk}\noindent
By  Proposition~\ref{energyderivative} (see also Remark~\ref{alltimes}) for a.e. $t\in\left[0,\frac{\ell_0}{2}\right]$ we have
\begin{align*}
	\dot{\mathcal{T}}(t;\tilde\ell,\tilde{w})&=-\frac{\dot{\tilde\ell}(t)}{2}\frac{1-\dot{\tilde\ell}(t)}{1+\dot{\tilde\ell}(t)}e^{-\nu t}\left[\dot v_0(\tilde\ell(t){-}t)-v_1(\tilde\ell(t){-}t)-\frac{\nu^2}{4}\int_{0}^{t}\tilde v(\tau,\tau{-}t{+}\tilde\ell(t))\d\tau\right]^2\\
	&\quad+\dot{\tilde w}(t)\left[\dot{\tilde w}(t)+\frac{\nu}{2}\tilde w(t)-e^{-\frac{\nu t}{2}}\left(\dot v_0(t)+v_1(t)+\frac{\nu^2}{4}\int_{0}^{t}\tilde v(\tau,t{-}\tau)\d\tau\right)\right],
\end{align*}
where $\tilde v(t,x)=e^{\nu t/2}\tilde u(t,x)$ and $v_0$ and $v_1$ are given by \eqref{datav}.\par
Since in \eqref{formal} we want to compute the right derivative of $\mathcal T(t; \tilde\ell,\tilde w)$ precisely at $t = \bar t$, we need a slight improvement of Proposition~\ref{energyderivative} (see Theorem~\ref{DERR} below and the analogous Proposition~2.1 in~\cite{DMLazNar16}). With this aim we will require that there exist $\alpha,\,\beta \in \erre$ such that
\begin{subequations}
\begin{equation}
	\label{alphaleb}
	\lim\limits_{h\to 0^+}\frac1h \int_{\bar t}^{\bar t+h} \left|\dot{\tilde\ell}(t) - \alpha  \right|  \d t =0 ,
\end{equation}
\begin{equation}
	\label{betagammaleb}
	\lim\limits_{h\to 0^+}\frac1h \int_{\bar t}^{\bar t+h} \left|\dot{\tilde w}(t) - \beta  \right|^2  \d t =0 .
\end{equation}
\end{subequations}
\begin{thm}\label{DERR}
	Fix $\nu\ge 0$, $\ell_0>0$ and consider $u_0$, $u_1$ and $w$ satisfying \eqref{bdryregularity} and \eqref{compatibility0}. Assume that $\ell\colon[0,+\infty)\to[\ell_0,+\infty)$ satisfies \eqref{ellea} and \eqref{ellemin}.\par 
	Then there exists a set $N\subseteq[0,+\infty)$ of measure zero, depending only on $\ell,\,u_0,\,u_1$ and $w$, such that for every $\bar{t}\in[0,+\infty)\setminus N$ the following statement holds true:\\
	if $v_0$, $v_1$, $\tilde{\ell}$, $\tilde{w}$, $\tilde u$, $\tilde v$, $u$ and $v$ are as above, 
	if $\dot{\tilde{\ell}}$ and $\dot{\tilde{w}}$ satisfy \eqref{alphaleb} and \eqref{betagammaleb} respectively, then
	\begin{equation*}
		\dot{\mathcal{T}}_r(\bar{t};\tilde{\ell},\tilde{w}):=\lim\limits_{h\to0^+}\frac{\mathcal{T}(\bar{t}+h;\tilde{\ell},\tilde{w})-\mathcal{T}(\bar{t};\tilde{\ell},\tilde{w})}{h}\quad\quad\text{exists.}
	\end{equation*}
	Moreover, if $\bar t\in\left[0,\frac{\ell_0}{2}\right]\setminus N$, one has the explicit formula
	\begin{align*}
		\dot{\mathcal{T}}_r(\bar{t};\tilde{\ell},\tilde{w})&=-\frac{\alpha}{2}\frac{1-\alpha}{1+\alpha}e^{-\nu\bar{t}}\left[\dot{v}_0(\ell(\bar t\,){-}\bar t\,)-v_1(\ell(\bar t\,){-}\bar t\,)-\frac{\nu^2}{4}\int_{0}^{\bar{t}}v(\tau,\tau{-}\bar{t}{+}\ell(\bar t\,))\d\tau\right]^2\\
		&\quad+\beta\left[\beta+\frac{\nu}{2}w(\bar t\,)-e^{-\frac{\nu\bar{t}}{2}}\left(\dot{v}_0(\bar t\,)+v_1(\bar t\,)+\frac{\nu^2}{4} \int_{0}^{\bar{t}}v(\tau,\bar{t}{-}\tau)\d\tau\right)\right].
	\end{align*}
\end{thm}
\begin{rmk}\label{alltimes2}
		One can obtain a similar formula for $\dot{\mathcal{T}}_r(\bar{t};\tilde{\ell},\tilde{w})$, valid for $\bar t\ge\frac{\ell_0}{2}$, reasoning as in Remark~\ref{alltimes}.
\end{rmk}
\begin{proof}[Proof of Theorem~\ref{DERR}]
	Let us define $T:=\ell_0/2$; we notice that by Remarks~\ref{alltimes} and \ref{alltimes2} it is enough to prove the Theorem in the time interval $[0,T]$.\par
	We call $\rho_1(r):=\dot{v}_0(r)-v_1(r)$ and $\rho_2(r):=\dot{v}_0(r)+v_1(r)$ and we consider the points $\bar{t}\in[0,T]$ with the following properties:
	\begin{itemize}
		\item[a)] $\displaystyle\lim\limits_{h\to0^+}\displaystyle\!\frac 1h \displaystyle\int_{\bar{t}{-}\ell(\bar t\,)}^{\bar{t}{-}\ell(\bar t\,)+h}\!\!\!\!\!\!\!\!\!\!\!\!\!\!\!\!|\left(\rho_1({-}r)\right)^2\!-\!\big(\rho_1(\ell(\bar t\,){-}\bar t\,)\big)^2|\d r\!=\!0\,$ and $\displaystyle\lim\limits_{h\to0^+}\!\displaystyle\frac 1h \displaystyle\int_{\bar{t}{-}\ell(\bar t\,)}^{\bar{t}{-}\ell(\bar t\,)+h}\!\!\!\!\!\!\!\!\!\!\!\!\!\!\!\!|\rho_1({-}r)\!-\!\rho_1(\ell(\bar t\,){-}\bar t\,)|\d r\!=\!0$;
		\item[b)] $\displaystyle\lim\limits_{h\to0^+}\displaystyle\frac 1h \displaystyle\int_{\bar{t}}^{\bar{t}+h}|\rho_2(r)-\rho_2(\bar t\,)|^2\d r=0.$
	\end{itemize}

We call $E_T$ the set of points satisfying a) and b). Since $\rho_1$ and $\rho_2$ belong to $L^2(0,\ell_0)$ and since $\ell$ satisfies \eqref{ellemin} the set $N_T:=[0,T]\setminus E_T$ has measure zero (see Corollary \ref{nullsets}). Let us fix $\bar{t}\in E_T$.\par
In the estimates below the symbol $C$ is used to denote a constant, which may change from line to line, that does not depend on $h$, although it can depend on $\bar t$.
For the sake of clarity we define $\displaystyle I_1(v,\ell)(t):=\frac{\nu^2}{4}\int_{0}^{t}v(\tau,\tau{-}t{+}\ell(t))\d\tau$ and $\displaystyle I_2(v)(t):=\frac{\nu^2}{4}\int_{0}^{t}v(\tau,t{-}\tau)\d\tau$, so that
\begin{align*}
&\left|\!\frac{\mathcal{T}(\bar{t}{+}{h};\tilde{\ell},\tilde{w}){-}\!\mathcal{T}(\bar{t};\tilde{\ell},\tilde{w})}{h}{+}\frac{\alpha}{2}\!\frac{1\!\!-\!\alpha}{1\!\!+\!\alpha}e^{\!{-}\nu\bar{t}}\Big[\!\rho_1(\ell(\bar t){-}\bar t){-}I_1(v,\ell)(\bar t)\!\Big]^2\!\!\!{-}
\beta\!\!\left[\!\beta{+}\frac{\nu}{2}w(\bar t){-}e^{{-}\frac{\nu\bar t}{2}}\!\Big(\!\rho_2(\bar t){+}I_2(v)(\bar t)\!\Big)\!\right]\!\right|\\
&\le\frac{1}{2h}\int_{\bar{t}}^{\bar{t}{+}h}\left|\dot{\tilde{\ell}}(s)\frac{1{-}\dot{\tilde{\ell}}(s)}{1{+}\dot{\tilde{\ell}}(s)}e^{\!{-}\nu s}\left[\rho_1(\tilde\ell(s){-}s){-}I_1(\tilde{v},\tilde{\ell})(s)\right]^2\!\!\!{-}\alpha\frac{1{-}\alpha}{1{+}\alpha}e^{{-}\nu\bar{t}}\Big[\rho_1(\ell(\bar t\,){-}\bar t\,){-}I_1(v,\ell)(\bar t\,)\Big]^2\right|\!\!\d s\\
&\,\,\,+\frac 1h\int_{\bar{t}}^{\bar{t}+{h}}\!\!\left|\dot{\tilde{w}}(s)\!\!\left[\dot{\tilde{w}}(s){+}\frac{\nu}{2}\tilde{w}(s){-}e^{\!{-}\frac{\nu s}{2}}\Big(\rho_2(s){+}I_2(\tilde{v})(s)\Big)\!\right]\!\!{-}\beta\!\left[\!\beta{+}\frac{\nu}{2}w(\bar t\,){-}e^{\!{-}\frac{\nu\bar t}{2}}\Big(\rho_2(\bar t\,){+}I_2(v)(\bar t\,)\Big)\!\right]\!\right|\!\!\d s .
\end{align*}
We denote by $J_1$ and $J_2$ the first and the second integral respectively and we estimate:
\begingroup
\allowdisplaybreaks
\begin{align*}
	J_1&\le e^{-\nu\bar{t}}\Big[\rho_1(\ell(\bar t\,){-}\bar t\,)-I_1(v,\ell)(\bar t\,)\Big]^2\frac{1}{2h}\int_{\bar{t}}^{\bar{t}+h}\left|\dot{\tilde{\ell}}(s)\frac{1-\dot{\tilde{\ell}}(s)}{1+\dot{\tilde{\ell}}(s)}-\alpha\frac{1-\alpha}{1+\alpha}\right|\d s\\	
	&\quad+\Big[\rho_1(\ell(\bar t\,){-}\bar t\,)-I_1(v,\ell)(\bar t\,)\Big]^2\frac{1}{2h}\int_{\bar{t}}^{\bar{t}+h}\dot{\tilde{\ell}}(s)\frac{1-\dot{\tilde{\ell}}(s)}{1+\dot{\tilde{\ell}}(s)}|e^{-\nu s}-e^{-\nu\bar{t}}|\d s\\
	&\quad+\frac{1}{2h}\int_{\bar{t}}^{\bar{t}+h}\dot{\tilde{\ell}}(s)\frac{1-\dot{\tilde{\ell}}(s)}{1+\dot{\tilde{\ell}}(s)}e^{-\nu s}\left|\left[\rho_1(\tilde\ell(s){-}s)-I_1(\tilde{v},\tilde{\ell})(s)\right]^2-\Big[\rho_1(\ell(\bar t\,){-}\bar t\,)-I_1(v,\ell)(\bar t\,)\Big]^2\right|\d s\\
	&\le\frac{C}{h}\int_{\bar{t}}^{\bar{t}+h}|\dot{\tilde{\ell}}(s)-\alpha|\d s+\frac{C}{h}\int_{\bar{t}}^{\bar{t}+h}|e^{-\nu s}-e^{-\nu\bar{t}}|\d s\\
	&\quad+\frac{1}{2h}\int_{\bar{t}}^{\bar{t}+h}(1-\dot{\tilde{\ell}}(s))\left|\left[\rho_1(\tilde\ell(s){-}s)-I_1(\tilde{v},\tilde{\ell})(s)\right]^2-\Big[\rho_1(\ell(\bar t\,){-}\bar t\,)-I_1(v,\ell)(\bar t\,)\Big]^2\right|\d s.
\end{align*}
\endgroup
The first two integrals vanish as $h\to 0^+$, so we only need to estimate the last integral, denoted by $\tilde{J}_1$:
\begingroup
\allowdisplaybreaks
\begin{align*}
	\tilde{J}_1&\le\frac{1}{2h}\!\int_{\bar{t}}^{\bar{t}+h}\!\!\!\!\!\!\!\!\!(1-\dot{\tilde{\ell}}(s))\!\left|\!\left(\rho_1(\tilde\ell(s){-}s)\right)^2\!\!\!{-}\big(\rho_1(\ell(\bar t\,){-}\bar t\,)\big)^2\right|\!\!\d s{+}\frac{1}{2h}\int_{\bar{t}}^{\bar{t}+h}\left|\left(I_1(\tilde{v},\tilde{\ell})(s)\right)^2\!\!\!{-}\big(I_1(v,\ell)(\bar t\,)\big)^2\right|\!\!\d s\\
	&\quad+\frac 1h \int_{\bar{t}}^{\bar{t}+h}(1-\dot{\tilde{\ell}}(s))\left|\rho_1(\tilde\ell(s){-}s)-\rho_1(\ell(\bar t\,){-}\bar t\,)\right||I_1(\tilde{v},\tilde{\ell})(s)|\d s\\
	&\quad +\frac{|\rho_1(\ell(\bar t\,){-}\bar t\,)|}{h} \int_{\bar{t}}^{\bar{t}+h}\left|I_1(\tilde{v},\tilde{\ell})(s)-I_1(v,\ell)(\bar t\,)\right|\d s\\
	&\le\frac{1}{2h}\int_{\bar{t}{-}\ell(\bar t\,)}^{\bar{t}{-}\ell(\bar t\,)+h}\left|\left(\rho_1({-}r)\right)^2\!{-}\big(\rho_1(\ell(\bar t\,){-}\bar t\,)\big)^2\right|\d r+\frac{1}{2h}\int_{\bar{t}}^{\bar{t}+h}\left|\left(I_1(\tilde{v},\tilde{\ell})(s)\right)^2\!\!{-}\big(I_1(v,\ell)(\bar t\,)\big)^2\right|\d s\\
	&\quad+\frac{C}{h}\int_{\bar{t}{-}\ell(\bar t\,)}^{\bar{t}{-}\ell(\bar t\,)+h}\left|\rho_1({-}r)-\rho_1(\ell(\bar t\,){-}\bar t\,)\right|\d r+\frac{C}{h} \int_{\bar{t}}^{\bar{t}+h}\left|I_1(\tilde{v},\tilde{\ell})(s)-I_1(v,\ell)(\bar t\,)\right|\d s.
\end{align*}
\endgroup
The first and the third integral tend to $0$ when $h\to0^+$ by assumption a), while the other two by the continuity of the function $I_1(\tilde{v},\tilde{\ell})$. Now we estimate $J_2$:
\begin{align*}
	J_2&\le\frac 1h\int_{\bar{t}}^{\bar{t}+h}\!\!\!\!\!\!\!\!|\dot{\tilde{w}}(s)|\left|\dot{\tilde{w}}(s){+}\frac{\nu}{2}\tilde{w}(s){-}e^{-\frac{\nu s}{2}}\Big(\rho_2(s)+I_2(\tilde{v})(s)\Big){-}\beta{-}\frac{\nu}{2}w(\bar t\,){+}e^{-\frac{\nu\bar{t}}{2}}\Big(\rho_2(\bar t\,)+I_2(v)(\bar t\,)\Big)\right|\d s\\
	&\quad+\left|\beta+\frac{\nu}{2}w(\bar t\,)-e^{-\frac{\nu\bar{t}}{2}}\Big(\rho_2(\bar t\,)+I_2(v)(\bar t\,)\Big)\right|\frac 1h\int_{\bar{t}}^{\bar{t}+h}|\dot{\tilde{w}}(s)-\beta|\d s\\
	&\le\frac 1h\int_{\bar{t}}^{\bar{t}+h}|\dot{\tilde{w}}(s)||\dot{\tilde{w}}(s)-\beta|\d s+\frac{\nu}{2h}\int_{\bar{t}}^{\bar{t}+h}|\dot{\tilde{w}}(s)||\tilde{w}(s)-w(\bar t\,)|\d s+\frac{C}{h}\int_{\bar{t}}^{\bar{t}+h}|\dot{\tilde{w}}(s)-\beta|\d s\\
	&\quad+\frac 1h\int_{\bar{t}}^{\bar{t}+h}|\dot{\tilde{w}}(s)|\left|e^{-\frac{\nu s}{2}}\Big(\rho_2(s)+I_2(\tilde{v})(s)\Big)-e^{-\frac{\nu\bar{t}}{2}}\Big(\rho_2(\bar t\,)+I_2(v)(\bar t\,)\Big)\right|\d s.
\end{align*}
The first three integrals tend to $0$ as $h\to 0^+$ since $\lim\limits_{h\to 0^+}\frac 1h\int_{\bar{t}}^{\bar{t}+h}|\dot{\tilde{w}}(s)-\beta|^2\d s=0$ and by the continuity of $\tilde{w}$, so we only need to estimate the last one, denoted by $\tilde{J}_2$:
\begin{align*}
	\tilde{J}_2&\le\frac 1h\int_{\bar{t}}^{\bar{t}+h}e^{-\frac{\nu s}{2}}|\dot{\tilde{w}}(s)||\rho_2(s)-\rho_2(\bar t\,)|\d s+\frac 1h\int_{\bar{t}}^{\bar{t}+h}e^{-\frac{\nu s}{2}}|\dot{\tilde{w}}(s)||I_2(\tilde{v})(s)-I_2(v)(\bar t\,)|\d s\\
	&\quad+|\rho_2(\bar t\,)+I_2(v)(\bar t\,)|\frac 1h\int_{\bar{t}}^{\bar{t}+h}|\dot{\tilde{w}}(s)||e^{-\frac{\nu s}{2}}-e^{-\frac{\nu\bar{t}}{2}}|\d s.
\end{align*}
Exploiting assumption b) and the continuity of $I_2(\tilde{v})$ we conclude.
\end{proof}
Thanks to Theorem~\ref{DERR} we can give the rigorous definition of dynamic energy release rate:
\begin{defi}[\textbf{Dynamic energy release rate}]\label{galfa}
	Fix $\nu\ge0$, $\ell_0>0$ and consider $u_0$, $u_1$ and $w$ satisfying \eqref{bdryregularity} and \eqref{compatibility0}. Assume that $\ell\colon[0,+\infty)\to[\ell_0,+\infty)$ satisfies \eqref{ellea} and \eqref{ellemin}.\par
	For a.e. $ \bar{t}\in[0,+\infty)$ and for every  $\alpha\in(0,1)$ the dynamic energy release rate corresponding to the velocity $\alpha$ of the debonding front is defined as
	\begin{equation*}
		G_\alpha(\bar t\,) := -\frac{1}{\alpha}\dot{\mathcal{T}}_r(\bar{t};\tilde{\ell},\bar{w}),
	\end{equation*}
	where $\tilde{\ell}$ is an arbitrary Lipschitz extension of $\ell_{|_{[0,\bar{t}\,]}}$ satisfying \eqref{ellemin} and \eqref{alphaleb}, while 
	\begin{equation*}
		\bar{w}(t)=\begin{cases}
		w(t)&\mbox{if }t\in[0,\bar{t}\,],\\
		w(\bar t\,)&\mbox{if }t\in(\bar{t},+\infty).
		\end{cases}
	\end{equation*}
\end{defi}\noindent
By Theorem~\ref{DERR} for a.e. $\bar{t}\in\left[0,\frac{\ell_0}{2}\right]$ we get
\begin{equation}\label{derr}
	G_\alpha(\bar t\,) = \frac{1}{2}\frac{1-\alpha}{1+\alpha} e^{-\nu\bar{t}}\left[\dot{v}_0(\ell(\bar t\,){-}\bar t\,)-v_1(\ell(\bar t\,){-}\bar t\,)-\frac{\nu^2}{4}\int_{0}^{\bar{t}}v(\tau,\tau{-}\bar{t}{+}\ell(\bar t\,))\d\tau\right]^2,
\end{equation}
and a similar formula holds true for a.e. $\bar t\ge \frac{\ell_0}{2}$ by Remarks~\ref{alltimes} and \ref{alltimes2}. In the case $\nu=0$ we have the expression
\begin{equation}\label{previous}
	G_\alpha(\bar t\,) = 2\frac{1-\alpha}{1+\alpha} \left[\frac{\dot{u}_0(\ell(\bar t\,){-}\bar t\,)-u_1(\ell(\bar t\,){-}\bar t\,)}{2}\right]^2,\quad\quad\mbox{for a.e. }\bar{t}\in\left[0,\frac{\ell_0}{2}\right],
\end{equation} 
and hence we recover the formula given in~\cite{DMLazNar16}.\par 
We then extend the dynamic energy release rate to the case $\alpha=0$ by continuity, so that
\begin{equation*}
	G_\alpha(\bar t\,)=\frac{1-\alpha}{1+\alpha}G_0(\bar t\,),\quad\quad \text{ for a.e. }\bar{t}\in[0,+\infty).
\end{equation*}
In particular by \eqref{derr} we know that for a.e. $\bar t\in\left[0,\frac{\ell_0}{2}\right]$ we can write
\begin{equation}\label{G0}
	G_0(\bar t\,) = \frac{1}{2}e^{-\nu\bar{t}}\left[\dot{v}_0(\ell(\bar t\,){-}\bar t\,)-v_1(\ell(\bar t\,){-}\bar t\,)-\frac{\nu^2}{4}\int_{0}^{\bar{t}}v(\tau,\tau{-}\bar{t}{+}\ell(\bar t\,))\d\tau\right]^2.
\end{equation}
We want to highlight that in the damped case $\nu>0$ the dynamic energy release rate depends directly on $v$ and $\ell$, see \eqref{derr}, while in the undamped one it depends only on the debonding front $\ell$ (at least for small times), see \eqref{previous}. This is the main reason why the arguments used in \cite{DMLazNar16} become useless if viscosity is taken into account and new ideas have to be developed.
\subsection{Griffith's criterion}
To introduce the criterion which controls the evolution of the debonding front $\ell$ we need to consider the notion of local toughness of the glue between the substrate and the tape. It is a measurable function $\kappa \colon [\ell_0,+\infty) \to (0,+\infty)$  which rules the amount of energy dissipated during the debonding process in the time interval $[0,t]$ via the formula
\begin{equation}\label{enkappa}
	\int_{\ell_0}^{\ell(t)} \kappa(x) \d x.
\end{equation}
As in~\cite{DMLazNar16} and~\cite{Lar10} we postulate that our model is governed by an energy-dissipation balance and a maximum dissipation principle; this last one states that the debonding front has to move with the maximum speed allowed by the energy balance. More precisely we assume: 
	\begin{equation}\label{enbal}
	\mc{T}(t)+\int_{\ell_0}^{\ell(t)}\kappa(x)\d x = \mc{T}(0) +\mathcal{W}(t),\quad\quad\mbox{ for every }t\in[0,+\infty),
	\end{equation} 
	\begin{equation}\label{dissipprinc}
	\dot \ell(t) = \max\{\alpha\in[0,1) \mid \kappa(\ell(t))\alpha = G_\alpha(t) \alpha\},\quad\quad\text{for a.e. } t\in[0,+\infty),
	\end{equation}
where $\mathcal{W}$ is the work of the external loading and it has the form (see also Remark~\ref{alltimes}):
\begin{equation*}
\mathcal{W}(t):=\int_{0}^{t}\!\dot w(s)\!\left[\dot w(s){+}\frac{\nu}{2}w(s){-}e^{-\frac{\nu s}{2}}\!\left(\!\dot v_0(s){+}v_1(s){+}\frac{\nu^2}{4}\!\int_{0}^{s}\!v(\tau,s{-}\tau)\d\tau\!\right)\!\right]\!\!\d s,\,\,\text{ for }t\in\left[0,\frac{\ell_0}{2}\right].
\end{equation*} 
By Proposition~\ref{energyderivative}, Theorem~\ref{DERR} and Lemma~\ref{areaformula} we deduce that \eqref{enbal} is equivalent to
\begin{equation*}
	\kappa(\ell(t)) \dot \ell(t) = G_{\dot \ell(t)}(t) \dot \ell(t), \quad \text{for a.e.}\ t\in[0,+\infty),
\end{equation*}
and we observe that for a.e. $t\in[0,+\infty)$ the set $\{\alpha\in[0,1) \mid \kappa(\ell(t))\alpha = G_\alpha(t) \alpha\}$ has at most one element different from zero by the strict monotonicity of $\alpha \mapsto G_\alpha(t)$ and since $\kappa(x)>0$ for every $x\ge\ell_0$. Therefore the maximum dissipation principle \eqref{dissipprinc} simply states that during the evolution of the debonding front $\ell$ only two phases can occur: if the toughness $\kappa$ is strong enough, $\ell$ stops and does not move till the dynamic energy release rate equals $\kappa$, otherwise it moves at the only speed which is consistent with the energy-dissipation balance \eqref{enbal}.

Arguing as in~\cite{DMLazNar16} we get that \eqref{enbal}\&\eqref{dissipprinc} are equivalent to the following system, called Griffith's criterion in analogy to the corresponding criterion in Fracture Mechanics:  
\begin{equation}\label{Griffithcrit}
	\begin{cases}
	\,\,0\le\dot{\ell}(t)<1,\\ 
	\,\,G_{\dot\ell(t)}(t)\le \kappa(\ell(t)),\\ 
	\left[ G_{\dot\ell(t)}(t)-\kappa(\ell(t))\right]\dot{\ell}(t)=0, 
	\end{cases}\quad\quad\quad\quad \text{for a.e. } t\in[0,+\infty).
\end{equation}

Finally one can prove (see~\cite{DMLazNar16} for more details) that Griffith's criterion \eqref{Griffithcrit} is equivalent to the following ordinary differential equation:
\begin{equation}\label{equation}
\dot{\ell}(t)=\max\left\{\frac{G_0(t)-\kappa(\ell(t))}{G_0(t)+\kappa(\ell(t))},0\right\},\quad\quad\quad\quad \text{for a.e. } t\in[0,+\infty),
\end{equation}
which, for a.e. $t\in\left[0,\frac{\ell_0}{2}\right]$, can be rewritten as
\begin{equation}
	\label{ellepunto}
	\dot{\ell}(t)=\max\left\{\frac{\left[\dot{v}_0(\ell(t){-}t)-v_1(\ell(t){-}t)-\displaystyle\frac{\nu^2}{4}\int_{0}^{t}v(\tau,\tau{-}t{+}\ell(t))\d\tau\right]^2-2e^{\nu t}\kappa(\ell(t))}{\left[\dot{v}_0(\ell(t){-}t)-v_1(\ell(t){-}t)-\displaystyle\frac{\nu^2}{4}\int_{0}^{t}v(\tau,\tau{-}t{+}\ell(t))\d\tau\right]^2+2e^{\nu t}\kappa(\ell(t))},0\right\}.
\end{equation}
We want to underline again that, differently from \cite{DMLazNar16}, the equation for the debonding front \eqref{ellepunto} depends also on $v$ (and thus on $u$) if $\nu>0$. This will bring the main technical difficulties of the next Section.
\section{Evolution of the debonding front}	\label{sec4}
In this Section we couple problem \eqref{problem1} with the energy-dissipation balance \eqref{enbal} and the maximum dissipation principle \eqref{dissipprinc} and we prove existence of a unique pair $(u,\ell)$ which solves this coupled problem.\par
We fix $\nu\ge 0$, $\ell_0>0$ and we consider $u_0$, $u_1$ and $w$ satisfying \eqref{bdryregularity} and \eqref{compatibility0}, and a measurable function  $\kappa\colon[\ell_0,+\infty)\to(0,+\infty)$.\par
Since differently from previous Sections the debonding front $\ell$ is unknown, from now on we will always stress the dependence on $\ell$ and we shall write $A_\ell$, $R_\ell$ and $\Omega_\ell$ instead of $A$, $R$ and $\Omega$, and so on. We shall also write $(G_0)_{v,\ell}$ instead of $G_0$, since by \eqref{G0} the dependence of the dynamic energy release rate both on the debonding front $\ell$ and on the solution $v$ of \eqref{problem2} is evident. Moreover, as in Lemmas~\ref{A} and \ref{derivatives}, we shall extend the functions $A_\ell$ and $\displaystyle\iint_{R_\ell(\cdot,\cdot)}v\d\sigma\d\tau$ setting them to be equal $0$ outside $\overline{\Omega_\ell}$.
\begin{defi}\label{coupleddef}
	Assume $\ell\colon[0,+\infty)\to[\ell_0,+\infty)$ satisfies \eqref{ellea} and \eqref{ellemin}; let $u\colon[0,+\infty)^2\to\erre$ be such that $u\in\widetilde{H}^1(\Omega_\ell)$ (resp. in $H^1((\Omega_\ell)_T)$). We say that the pair $(u,\ell)$ is a solution of the coupled problem (resp. in $[0,T]$) if:
	\begin{itemize}
		\item[i)] $u$ solves problem \eqref{problem1} in $\Omega_\ell$ (resp. in $(\Omega_\ell)_T)$ in the sense of Definition \ref{sol},
		\item[ii)] $u\equiv 0$ outside $\overline{\Omega_\ell}$ (resp. in $([0,T]\!\times\![0,+\infty))\setminus\overline{(\Omega_\ell)_T}$),
		\item[iii)]  $(u,\ell)$ satisfies Griffith's criterion \eqref{Griffithcrit} for a.e. $t\in[0,+\infty)$ (resp. for a.e. $t\in[0,T]$).
	\end{itemize}	
\end{defi} 
	Using \eqref{problem2} and \eqref{equation} it turns out that the pair $(u,\ell)$ is a solution of the coupled problem if and only if $(v,\ell)$, where $v(t,x)=e^{\nu t/2}u(t,x)$, satisfies the following system:
	\begin{equation}\label{coupledv}
	\begin{cases}
	v_{tt}(t,x)-v_{xx}(t,x)-\frac{\nu^2}{4} v(t,x)=0, &t>0\,,\,0<x<\ell(t),\\
	\dot{\ell}(t)=\displaystyle\max\left\{\frac{(G_0)_{v,\ell}(t)-\kappa(\ell(t))}{(G_0)_{v,\ell}(t)+\kappa(\ell(t))},0\right\}, &t>0,\\
	v(t,x)=0, &t>0\,,\,x>\ell(t),\\
	v(t,0)=z(t), &t>0,\\
	v(t,\ell(t))=0,&t>0,\\
	v(0,x)=v_0(x),&0<x<\ell_0,\\
	v_t(0,x)=v_1(x),&0<x<\ell_0,\\
	\ell(0)=\ell_0.
	\end{cases}		
	\end{equation}
	Similarly to Section~\ref{sec1} we write the fixed point problem related to \eqref{coupledv}. Since representation formula \eqref{Duhamel} holds true only in $\Omega'_\ell$, we fix $T\in\left(0,\frac{\ell_0}{2}\right)$ and we state the problem in $(\Omega_\ell)_T$:
	\begin{equation*}
	\begin{cases}
	v(t,x)=\left(A_\ell(t,x)+\displaystyle\frac{\nu^2}{8}\displaystyle\iint_{R_\ell(t,x)}\!\!\!\!\!\!\!\!\!\!v(\tau,\sigma)\d\sigma \d\tau\right)\chi_{(\Omega_\ell)_T}(t,x),&\text{for a.e. }(t,x)\in(0,T)\!\!\times\!\!(0,+\infty),\\
	\ell(t)=\ell_0+\displaystyle\int_{0}^{t}\max\left\{\frac{(G_0)_{v,\ell}(s)-\kappa(\ell(s))}{(G_0)_{v,\ell}(s)+\kappa(\ell(s))},0\right\}\d s,&\text{for every }t\in[0,T],
	\end{cases}		
	\end{equation*}
	where, given a set $E$, we denoted by $\chi_E$ the indicator function of $E$.\par
	For a reason that will be clear later we prefer to introduce the auxiliary function $\lambda$, defined as the inverse of the map $t\mapsto t{-}\ell(t)$ (see also~\cite{DMLazNar16}, Theorem~3.5). We notice that $\lambda$ is absolutely continuous by \eqref{ellemin} and Corollary \ref{absinv}, while in the simpler case in which there exists $\delta_T\in(0,1)$ such that $0\le \dot\ell(t)\le1-\delta_T$ for a.e. $t\in[0,T]$, $\lambda$ is Lipschitz and $1\le\dot\lambda(y)\le\frac{1}{\delta_T}$ for a.e. $y\in[-\ell_0,\lambda^{{-}1}(T)]$. We then consider the equivalent fixed point problem for the pair $(v,\lambda)$; exploiting \eqref{ellepunto} it takes the form:
	\begin{equation}\label{fixedv}
	\begin{cases}
	v(t,x)\!=\!\left(\!\!A_{\ell_\lambda}(t,x)+\displaystyle\frac{\nu^2}{8}\!\!\displaystyle\iint_{R_{\ell_\lambda}(t,x)}\!\!\!\!\!\!\!\!\!\!\!\!\!\!\!\!v(\tau,\sigma)\d\sigma \d\tau\!\right)\!\chi_{(\Omega_{\ell_\lambda})_T}(t,x),\!\!\!\!&\text{for a.e. }(t,x)\in(0,T)\!\!\times\!\!(0,+\infty),\\
	\lambda(y)=\displaystyle\frac 12\int_{{-}\ell_0}^{y}\Big(1+\max\left\{\Lambda_{v,\lambda}(s),1\right\}\Big)\d s,&\text{for every }y\in[-\ell_0,\lambda^{{-}1}(T)],
	\end{cases}		
	\end{equation}
	where we define for a.e. $y\in[-\ell_0,\lambda^{{-}1}(T)]$
	\begin{equation}\label{Lambda}
		\Lambda_{v,\lambda}(y):=\frac{\left[\dot{v}_0({-}y)-v_1({-}y)-\displaystyle\frac{\nu^2}{4}\int_{0}^{\lambda(y)}v(\tau,\tau{-}y)\d\tau\right]^2}{\displaystyle 2e^{\nu\lambda(y)}\kappa(\lambda(y){-}y)},
	\end{equation}
	and where we denoted by $\ell_\lambda$ simply the function $\ell$, stressing the fact that it depends on $\lambda$ via the formula $\ell_\lambda(t)=t-\lambda^{-1}(t)$.\par 
	As in Section~\ref{sec1}, we solve problem \eqref{fixedv} showing that a suitable operator is a contraction. We argue as follows: for $T>0$ and $Y\in(0,\ell_0)$ we consider the sets (see Figure~\ref{FigQ})
	\begin{gather*}
		Q=Q(T,Y):=\left\{(t,x)\mid 0\le t\le T,\,\ell_0{-}Y+t\le x\le \ell_0+t\right\},\\
		Q_{\ell_\lambda}:=Q\cap \overline{\Omega_{\ell_\lambda}}.
	\end{gather*}
	\begin{figure}
		\centering
		\includegraphics[scale=.8]{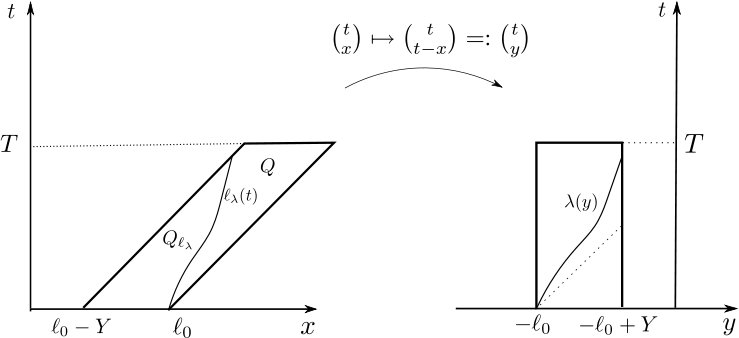}
		\caption{The set $Q$ and the functions $\lambda$ and $\ell_\lambda$.}\label{FigQ}
	\end{figure}\noindent	
	Moreover for $M>0$ and denoting by $I_Y$ the closed interval $[{-}\ell_0,{-}\ell_0{+}Y]$ we introduce the spaces 
	\begin{gather*}
		\mc X_1=\mc X_1(T,Y,M):=\left\{v\in C^0(Q)\mid \Vert v\Vert_{C^0(Q)}\le M\right\},\\
		\mc B_2=\mc B_2(T,Y):=\left\{\lambda\in C^0(I_Y)\mid \lambda(-\ell_0)=0,\, \Vert\lambda\Vert_{C^0(I_Y)}\,\le T,\,y\mapsto\lambda(y){-}y\text{ is nondecreasing}\right\}.
	\end{gather*}		
	Let us define $\mc X:=\mc X_1\times \mc B_2$ and consider the operators:
	\begin{subequations}
		\begin{equation*}
			\Psi_1(v,\lambda)(t,x):=\left(A_{\ell_\lambda}(t,x)+\frac{\nu^2}{8} \iint_{R_{\ell_\lambda}(t,x)}v(\tau,\sigma)\d\sigma\d\tau\right)\chi_{Q_{\ell_\lambda}}(t,x),
		\end{equation*}
		\begin{equation*}
		\Psi_2(v,\lambda)(y):=\frac 12\int_{-\ell_0}^{y}\!\left(\!1+\max\left\{\frac{\left[\dot{v}_0({-}s)-v_1({-}s)-\frac{\nu^2}{4}\int_{0}^{\lambda(s)}v(\tau,\tau{-}s)\d\tau\right]^2}{2e^{\nu\lambda(s)}\kappa(\lambda(s){-}s)},1\right\}\!\right)\!\d s.
		\end{equation*}
	\end{subequations}
We then define
\begin{equation}\label{Psi}
\Psi(v,\lambda):=\big(\Psi_1(v,\lambda),\Psi_2(v,\lambda)\big).
\end{equation}
\begin{rmk}
	From now on we shall write $\ell$, $\psi$ and $\omega$ instead of $\ell_\lambda$, $\psi_\lambda$ and $\omega_\lambda$, being tacit the dependence on $\lambda$.
\end{rmk} \noindent
For convenience, we assume for the moment that there exist two positive constants $c_1$ and $c_2$ such that
\begin{equation}\label{kappa}
0<c_1\le\kappa(x)\le c_2 \quad\mbox{for every } x\ge\ell_0.
\end{equation}
\begin{lemma}\label{itself}
	Fix $\nu\ge0$, $\ell_0>0$ and consider $v_0$, $v_1$ satisfying \eqref{bdry2} and $v_0(\ell_0)=0$. Assume that the measurable function $\kappa\colon[\ell_0,+\infty)\to(0,+\infty)$ satisfies \eqref{kappa}.\par
	Then for every $T>0$ and $M>0$ there exists $Y\in(0,\ell_0)$ such that the operator $\Psi$ in \eqref{Psi} maps $\mc X$ into itself.
\end{lemma}
\begin{proof} 
	Fix $T>0$, $M>0$ and let $(v,\lambda)\in \mc X$; by Lemmas~\ref{A} and \ref{derivatives} we deduce that $\Psi_1(v,\lambda)$ is continuous on $Q$ (indeed notice that $\ell=\ell_\lambda$ satisfies \eqref{elle}), while by construction $\Psi_2(v,\lambda)$ is actually absolutely continuous on $I_Y$ and satisfies $\Psi_2(v,\lambda)(-\ell_0)=0$ and $\frac{\d}{\d y}\Psi_2(v,\lambda)(y)\ge 1$ for a.e. $y\in I_Y$. Hence to conclude it is enough to find $Y\in(0,\ell_0)$ such that 
	\begin{equation*}
		\Vert\Psi_1(v,\lambda)\Vert_{C^0(Q)}\le M\quad\text{ and }\quad \Psi_2(v,\lambda)({-}\ell_0{+}Y)\le T.
	\end{equation*}
	We pick $(t,x)\in Q_\ell$ and using \eqref{homsol} we estimate:
	\begin{align*}
		|\Psi_1(v,\lambda)(t,x)|&\le |A_\ell(t,x)|+\frac{\nu^2}{8}\left|\iint_{R_{\ell}(t,x)}v(\tau,\sigma)\d\sigma\d\tau\right|\\
		&\le\int_{\ell_0{-}Y}^{\ell_0}\!\!\!\!\!\!\!\!\!(|\dot{v}_0(s)|+|v_1(s)|)\d s+\frac{\nu^2}{8}M|Q|\\
		&=\int_{\ell_0{-}Y}^{\ell_0}\!\!\!\!\!\!\!\!\!(|\dot{v}_0(s)|+|v_1(s)|)\d s+\frac{\nu^2}{8}MTY.
	\end{align*}
	As regards $\Psi_2(v,\lambda)({-}\ell_0{+}Y)$ we argue as follows:
	\begin{align*}
		\Psi_2(v,\lambda)({-}\ell_0{+}Y)&=\frac 12 \int_{-\ell_0}^{{-}\ell_0{+}Y}\big(1+\max\{\Lambda_{v,\lambda}(s),1\}\big)\d s\le\frac 12 \int_{-\ell_0}^{{-}\ell_0{+}Y}\big(2+\Lambda_{v,\lambda}(s)\big)\d s\\
		&\le Y+\frac{1}{2c_1}\left[\int_{-\ell_0}^{{-}\ell_0{+}Y}\!\!\!\!\!\!\!\!\!\!\!\![\dot{v}_0({-}s)-v_1({-}s)]^2\d s+\frac{\nu^4}{16}\int_{-\ell_0}^{{-}\ell_0{+}Y}\!\!\!\left(\int_{0}^{\lambda(s)}\!\!\!\!\!\!v(\tau,\tau{-}s)\d\tau\right)^2\!\!\!\d s\right]\\
		&\le Y+\frac{1}{2c_1}\int_{\ell_0{-}Y}^{\ell_0}\!\!\!\!\!\![\dot{v}_0(s)-v_1(s)]^2\d s+\frac{\nu^4}{32c_1}M^2T^2Y.		
	\end{align*}
	Since in both estimates the last line tends to $0$ when $Y\to0^+$ we can conclude.
\end{proof}
\begin{lemma}\label{equicontinuity}
	Fix $\nu\ge0$, $\ell_0>0$ and consider $v_0$, $v_1$ satisfying \eqref{bdry2} and $v_0(\ell_0)=0$. Fix $T>0$, $M>0$ and let $Y\in(0,\ell_0)$ be given by Lemma~\ref{itself}.\par 
	Then $\Psi_1(\mc X)$ is an equicontinuous family of $\mc X_1$.
\end{lemma}
\begin{proof}
	Let $(v,\lambda)\in \mc X$ and fix $\varepsilon>0$.\par 
	By simple geometric considerations and by continuity we deduce that
	\begin{itemize}
		\item[1)] $|R_\ell(t_1,x_1)\triangle R_\ell(t_2,x_2)|\!\le\!\frac{\sqrt{2}}{2}(4T+Y)\sqrt{|t_1{-}t_2|^2{+}|x_1{-}x_2|^2}$ for every $(t_1,x_1)$, $(t_2,x_2)\in Q_\ell$,
		\item[2)] there exists $\delta_1>0$ such that for every $a,\,b\in[0,\ell_0]$ satisfying $|a-b|\le\delta_1$ it holds
		\begin{equation*}
			|v_0(a){-}v_0(b)|+\left|\int_{a}^{b}v_1(r)\d r\right|\le\frac \varepsilon 2.
		\end{equation*}
	\end{itemize}
	Let us define $\displaystyle \delta:=\min\left\{\frac{\delta_1}{2}, \frac{4\sqrt{2}\,\,\varepsilon}{\nu^2M(4T+Y)}\right\}$ $\left(\displaystyle \delta=\frac{\delta_1}{2}\text{ if }\nu=0\right)$ and take $(t_1,x_1)$, $(t_2,x_2)\in Q$ satisfying $\sqrt{|t_1{-}t_2|^2{+}|x_1{-}x_2|^2}\le\delta$.\\ 
	For the sake of clarity we define $\displaystyle H_{v,\lambda}(t,x):= \left(\iint_{R_\ell(t,x)}v(\tau,\sigma)\d\sigma\d\tau\right)\chi_{Q_\ell}(t,x)$, so that
	\begin{align*}
		&|\Psi_1(v,\lambda)(t_1,x_1)-\Psi_1(v,\lambda)(t_2,x_2)|\\
		&\le|A_\ell(t_1,x_1)\chi_{Q_\ell}(t_1,x_1)-A_\ell(t_2,x_2)\chi_{Q_\ell}(t_2,x_2)|+\frac{\nu^2}{8}|H_{v,\lambda}(t_1,x_1)-H_{v,\lambda}(t_2,x_2)|=:I+II.
	\end{align*}
We notice that since $A_\ell\chi_{Q_\ell}$ and $H_{v,\lambda}$ vanish on $Q\setminus Q_\ell$ and they are continuous on the whole $Q$, it is enough to consider the case in which both $(t_1,x_1)$ and $(t_2,x_2)$ are in $Q_\ell$; in this case to estimate $II$ we use 2):
	\begin{align*}
		II\le\frac{\nu^2}{8}\iint_{R_{\ell}(t_1,x_1)\triangle R_{\ell}(t_2,x_2)}\!\!\!\!\!\!\!\!\!\!\!\!\!\!\!\!\!\!\!\!\!\!\!\!\!\!\!\!\!\!\!\!\!\!\!\!|v(\tau,\sigma)|\d\sigma\d\tau\le\frac{\nu^2}{8}M|R_{\ell}(t_1,x_1)\triangle R_{\ell}(t_2,x_2)|\le\frac{\nu^2}{16}M\sqrt{2}(4T+Y)\delta\le\frac{\varepsilon}{2}.
	\end{align*}
	For $I$ we exploit the explicit expression of $A_\ell$ given by \eqref{homsol} and we consider three different cases: if $(t_1,x_1)$, $(t_2,x_2)\in Q_\ell\cap\{t+x\le\ell_0\}$ we have
	\begin{align*}
		I\le\frac 12|v_0(x_1{+}t_1){-}v_0(x_2{+}t_2)|+\frac 12\left|\int_{x_2{+}t_2}^{x_1{+}t_1}\!\!\!\!\!\!\!v_1(r)\d r\right|+\frac 12|v_0(x_1{-}t_1){-}v_0(x_2{-}t_2)|+\frac 12\left|\int_{x_2{-}t_2}^{x_1{-}t_1}\!\!\!\!\!\!\!v_1(r)\d r\right|,
	\end{align*}
	and since $|(x_1{\pm}t_1)-(x_2{\pm}t_2)|\le 2\sqrt{|t_1{-}t_2|^2{+}|x_1{-}x_2|^2}\le\delta_1$, by 1) we deduce $I\le\varepsilon/2$.\par 
	If instead $(t_1,x_1)$, $(t_2,x_2)\in Q_\ell\cap\{t+x\ge\ell_0\}$ we get
	\begin{align*}
	I&\le\frac 12|v_0({-}\omega(x_1{+}t_1)){-}v_0({-}\omega(x_2{+}t_2))|+\frac 12\left|\int_{{-}\omega(x_2{+}t_2)}^{{-}\omega(x_1{+}t_1)}\!\!\!\!\!\!\!v_1(r)\d r\right|\\
	&\quad+\frac 12|v_0(x_1{-}t_1){-}v_0(x_2{-}t_2)|+\frac 12\left|\int_{x_2{-}t_2}^{x_1{-}t_1}\!\!\!\!\!\!\!v_1(r)\d r\right|,
	\end{align*}
	and since $|\omega(x_1{+}t_1)-\omega(x_2{+}t_2)|\le|(x_1{+}t_1)-(x_2{+}t_2)|\le \delta_1$ (we recall that $\omega$ is $1$-Lipschitz, see~\eqref{deromega}) again we have $I\le\varepsilon/2$.\par
	Finally if $(t_1,x_1)\in Q_\ell\cap\{t+x\le\ell_0\}$ while $(t_2,x_2)\in Q_\ell\cap\{t+x\ge\ell_0\}$ we get
	\begin{align*}
	I&\le\frac 12|v_0(x_1{-}t_1){-}v_0(x_2{-}t_2)|+\frac 12\left|\int_{x_2{-}t_2}^{x_1{-}t_1}\!\!\!\!\!\!\!v_1(r)\d r\right|\\
	&\quad+\frac 12|v_0(x_1{+}t_1){-}v_0({-}\omega(x_2{+}t_2))|+\frac 12\left|\int_{{-}\omega(x_2{+}t_2)}^{x_1{+}t_1}\!\!\!\!\!\!\!v_1(r)\d r\right|,
	\end{align*}
	and observing that for this configuration of $(t_1,x_1)$ and $(t_2,x_2)$ it holds 
	\begin{align*}
	|(x_1{+}t_1)+\omega(x_2{+}t_2)|&\le|(x_1{+}t_1)-\ell_0|+|\omega(x_2{+}t_2)-\omega(\ell_0)|\\
	&\le \ell_0-(x_1{+}t_1)+(x_2{+}t_2)-\ell_0\\
	&\le|t_1{-}t_2|+|x_1{-}x_2|\le\delta_1,
	\end{align*}
	we deduce also in this case $I\le\varepsilon/2$.\par 
	These estimates yield
	\begin{equation*}
		|\Psi_1(v,\lambda)(t_1,x_1)-\Psi_1(v,\lambda)(t_2,x_2)|\le I+II\le\varepsilon,
	\end{equation*} 
	and so we conclude.
\end{proof}
We now denote by $\mc B_1$ the closure of $\Psi_1(X)$ with respect to uniform convergence and we define $\mc B:=\mc B_1\times \mc B_2$; we notice that by Lemma~\ref{equicontinuity} and the Ascoli-Arzel\`a Theorem (see for istance~\cite{Rud}, Theorem~11.28) $\mc B$ is a complete metric space if endowed with the distance
\begin{equation}\label{distance}
\d\left((v^1,\lambda^1),(v^2,\lambda^2)\right):=\max\{\Vert v^1-v^2\Vert_{L^2(Q)}, \Vert \lambda^1-\lambda^2\Vert_{C^0(I_Y)}\}.
\end{equation}
\begin{prop}\label{contractioncoupled}
	Fix $\nu\ge0$, $\ell_0>0$ and consider $v_0$, $v_1$ satisfying \eqref{bdry2} and $v_0(\ell_0)=0$. Assume that $\kappa\in C^{0,1}([\ell_0,+\infty)) $ satisfies \eqref{kappa} and fix $T>0$ and $M>0$.\par
	Then there exists $Y\in(0,\ell_0)$ such that the operator $\Psi$ in \eqref{Psi} is a contraction from $(\mc B,\d)$ into itself.
\end{prop}
We prefer to postpone the (long and technical) proof of Proposition~\ref{contractioncoupled} to the end of the Section, so that we are at once in a position to state and prove the main result of the paper, which generalises Theorem~3.5 in~\cite{DMLazNar16}:
	\begin{thm}\label{finalthm}
		Fix $\nu\ge 0$, $\ell_0>0$ and consider $u_0$, $u_1$ and $w$ satisfying \eqref{bdryregularity} and \eqref{compatibility0}. Assume that the measurable function $\kappa\colon[\ell_0,+\infty)\to(0,+\infty) $ fulfills the following property:
		\begin{equation}\label{kappaeps}
			\mbox{for every }x\in[\ell_0,+\infty)\text{ there exists }\varepsilon=\varepsilon(x)>0\text{ such that } \kappa\in C^{0,1}([x,x+\varepsilon]).
		\end{equation}
		Then there exists a unique pair $(u,\ell)$ solving the coupled problem in the sense of Definition~\ref{coupleddef}. Moreover $u$ has a continuous representative on $\overline{\Omega_\ell}$ and it holds:
		\begin{equation*}
		u\in C^0([0,+\infty);H^1(0,+\infty))\cap C^1([0,+\infty);L^2(0,+\infty)). 
		\end{equation*}
	\end{thm}
\begin{rmk}
	Condition \eqref{kappaeps} allows for a wide range of left-discontinuous toughnesses, including $\kappa$ whose limits from the left (at discontinuity points) and to infinity can be $0$, $+\infty$ or they cannot even exist. However we point out that the right Lipschitzianity of $\kappa$ is instead crucial for the validity of the Theorem (see Remark~\ref{failureuniq}).
\end{rmk}
	\begin{proof}[Proof of Theorem~\ref{finalthm}]
		To conclude we need to prove there exists a unique pair $(v,\ell)$ solution of~\eqref{coupledv}. Rearranging Proposition~\ref{contractionCl} we firstly deduce there exists a unique $v^0$ satisfying \eqref{Duhamel} in the triangle $\{(t,x)\mid0\le t\le \ell_0,\,0\le x\le \ell_0{-}t\}$.\par
		Now consider $\varepsilon=\varepsilon(\ell_0)$ given by \eqref{kappaeps} and let us introduce a virtual toughness $\tilde{\kappa}$ which coincides with $\kappa$ in $[\ell_0,\ell_0+\varepsilon]$ and which is equal to $\kappa(\ell_0{+}\varepsilon)$ after $\ell_0+\varepsilon$. Since by construction $\tilde{\kappa}\in C^{0,1}([\ell_0,+\infty))$ and ${c_1}_\varepsilon\le\tilde{\kappa}(x)\le {c_2}_\varepsilon $ for some $0<{c_1}_\varepsilon\le{c_2}_\varepsilon$, exploiting Proposition~\ref{contractioncoupled} we can find $Y\in(0,\ell_0)$ and $T=T(Y)>0$ for which there exists a unique pair $(v^1,\ell^1)$ satisfying
		\begin{equation}\label{ciao}
			\begin{cases}
			v^1(t,x)=\left(A_{\ell^1}(t,x)+\displaystyle\frac{\nu^2}{8}\displaystyle\iint_{R_{\ell^1}(t,x)}v^1(\tau,\sigma)\d\sigma \d\tau\right)\chi_{Q_{\ell^1}}(t,x),&\mbox{for every }(t,x)\in Q,\\
			\ell^1(t)=\ell_0+\displaystyle\int_{0}^{t}\max\left\{\frac{(G_0)_{v^1,\ell^1}(s)-\tilde\kappa(\ell^1(s))}{(G_0)_{v^1,\ell^1}(s)+\tilde\kappa(\ell^1(s))},0\right\}\d s,&\mbox{for every } t\in[0,T].
			\end{cases}		
		\end{equation}
		Since $\ell^1(0)=\ell_0$ and $\tilde\kappa\equiv\kappa$ in $[\ell_0,\ell_0+\varepsilon]$, using the continuity of $\ell^1$ we deduce there exists a small time $T_\varepsilon>0$ such that $(v^1,\ell^1)$ satisfies \eqref{ciao} replacing $\tilde\kappa$ by $\kappa$ and $T$ by $T_\varepsilon$. Gluing together $v^0$ and $v^1$ and recalling Lemmas~\ref{A} and \ref{derivatives} we get the existence of a time $\widetilde{T}\in\left(0,\frac{\ell_0}{2}\right)$ satisfying the following properties:
		\begin{itemize}
			\item[a)] there exists a unique pair $(\tilde v,\tilde\ell)$ solution of \eqref{coupledv} in $[0,\widetilde T]$,
			\item[b)] $\tilde v$ belongs to $C^0([0,\widetilde{T}];H^1(0,+\infty))\cap C^1([0,\widetilde{T}];L^2(0,+\infty))$.						
		\end{itemize}
	
	Then we define $T^*:=\sup\{\widetilde{T}>0\mid \widetilde{T}\mbox{ satisfies a) and b)}\}$. If $T^*=+\infty$ we conclude; so let us assume by contradiction that $T^*<+\infty$ and consider an increasing sequence of times $\{T_k\}$ satisfying a) and  b) and converging to $T^*$. Let $(v_k,\ell_k)$ be the pair related to $T_k$ by a).\par
	Since by uniqueness $\ell_{k+1}(t)=\ell_{k}(t)$ for every $t\in[0,T_k]$ and since $0\le\dot\ell_k(t)< 1$ for a.e. $t\in[0,T_k]$, there exists a unique Lipschitz function $\ell$ defined on $[0,T^*]$ such that $\ell(t)=\ell_k(t)$ for every $t\in[0,T_k]$; hence $\ell(0)=\ell_0$ and $0\le\dot\ell(t)< 1$ for a.e. $t\in[0,T^*]$. Then by Theorem~\ref{globalsolutionprescr} there exists a unique continuous function $v$ on $\overline{(\Omega_\ell)_{T^*}}$ solution of $\eqref{problem2}$ in $(\Omega_\ell)_{T^*}$ belonging to $C^0([0,T^*];H^1(0,+\infty))\cap C^1([0,T^*];L^2(0,+\infty))$. Necessarily $v$ and $v_k$ coincide on $\overline{(\Omega_\ell)_{T_k}}$ for every $k\in\enne$ and hence $(v,\ell)$ is the unique solution of \eqref{coupledv} in $[0,T^*]$.\par
	Now we can repeat the contraction argument starting from time $T^*$: we replace $\ell_0$ by $\ell^*_0:=\ell(T^*)$, $v_0$ by $v(T^*,\cdot)\in H^1(0,\ell^*_0)$ and $v_1$ by $v_t(T^*,\cdot)\in L^2(0,\ell^*_0)$; notice that $v(T^*,0)=z(T^*)$ and $v(T^*,\ell^*_0)=0$, so the compatibility conditions \eqref{compatibility0} are satisfied. Arguing as before (now with $\varepsilon=\varepsilon(\ell_0^*)$ given by \eqref{kappaeps}) and as in the proof of Theorem~\ref{globalsolutionprescr} we deduce the existence of a time $\widehat{T}>T^*$ satisfying a) and b). This is absurd, being $T^*$ the supremum.
	\end{proof}
\begin{rmk}[\textbf{Regularity}]
	Arguing as in Remark~\ref{lipschitzregglob}, if we assume that $u_0\in C^{0,1}([0,\ell_0])$, $u_1\in L^\infty(0,\ell_0)$, $w\in\widetilde{C}{^{0,1}}([0,+\infty)) $ satisfy \eqref{compatibility0}, if the (measurable) toughness $\kappa$ satisfies \eqref{kappaeps}, then the solution $u$ belongs to $\widetilde{C}{^{0,1}}(\overline{\Omega_\ell})$ and $u_t(t,\cdot)$ is in $ L^\infty(0,+\infty)$ for every $t\in[0,+\infty)$. If in addition for every $\bar x>\ell_0$ there exists a positive constant $c_{\bar x}$ such that $\kappa(x)\ge c_{\bar x}$ for every $x\in[\ell_0,\bar x]$, then for every $T>0$ there exists $\delta_T\in(0,1)$ such that $0\le\dot\ell(t)\le 1-\delta_T$ for a.e. $t\in[0,T]$.
\end{rmk}
\begin{rmk}[\textbf{More regularity}]
	As in Remark~\ref{morereg}, if we assume that $u_0\in C^{1,1}([0,\ell_0])$, $u_1\in C^{0,1}([0,\ell_0])$, $w\in\widetilde{C}{^{1,1}}([0,+\infty)) $ satisfy \eqref{compatibility0}, if the toughness $\kappa\colon[\ell_0,+\infty)\to(0,+\infty)$ belongs to $\widetilde{C}{^{0,1}}([\ell_0,+\infty))$, in order to have $\ell\in \widetilde C^{1,1}([0,+\infty))$ and $u\in\widetilde C^{1,1}(\overline{\Omega_\ell})$ we need to impose a first order compatibility condition:
	\begin{equation}\label{boh}
	\begin{gathered}
	u_1(0)=\dot w(0),\\
	u_1(\ell_0)+\dot{u}_0(\ell_0)\max\left\{\frac{[\dot{u}_0(\ell_0)-u_1(\ell_0)]^2-2\kappa(\ell_0)}{[\dot{u}_0(\ell_0)-u_1(\ell_0)]^2+2\kappa(\ell_0)},0\right\}=0.
	\end{gathered}		
	\end{equation}
	Notice the relationship between \eqref{boh} and \eqref{comp2}, given by the equation for $\ell$ \eqref{ellepunto}. We want also to point out that the second condition in \eqref{boh} is equivalent to:
	\begin{equation*}
		\Big(u_1(\ell_0)=0,\,\, \dot{u}_0(\ell_0)^2\le 2\kappa(\ell_0)\Big)\text{ or }\Big(u_1(\ell_0)\neq 0,\,\,\dot{u}_0(\ell_0)^2{-}u_1(\ell_0)^2=2\kappa(\ell_0),\,\,\frac{\dot{u}_0(\ell_0)}{u_1(\ell_0)}<{-}1\Big).
	\end{equation*}
\end{rmk}
\begin{rmk}[\textbf{Time-dependent toughness}]
	Proposition~\ref{contractioncoupled}, and hence Theorem~\ref{finalthm}, holds true even in the case of a time-dependent toughness. To be precise, replacing \eqref{enkappa} by
	\begin{equation*}
		\int_{0}^{t}\kappa(s,\ell(s))\dot{\ell}(s)\d s,
	\end{equation*}
	where now $\kappa\colon[0,+\infty)\times[\ell_0,+\infty)\to(0,+\infty)$ also depends on time (and is Borel), we obtain that \eqref{equation} becomes
	\begin{equation*}
	\dot{\ell}(t)=\max\left\{\frac{G_0(t)-\kappa(t,\ell(t))}{G_0(t)+\kappa(t,\ell(t))},0\right\},\quad\quad\quad\quad \text{for a.e. } t\in[0,+\infty),
	\end{equation*}
	and in this case the denominator in \eqref{Lambda} reads as $2e^{\nu\lambda(y)}\kappa(\lambda(y),\lambda(y)-y)$.\par 
	So, if we assume that $\kappa\in C^{0,1}([0,+\infty)\times[\ell_0,+\infty))$ satisfies $0<c_1\le\kappa(t,x)\le c_2$ for every $(t,x)\in[0,+\infty)\times[\ell_0,+\infty)$, we can repeat with no changes the proofs of Lemma~\ref{itself} and Proposition~\ref{contractioncoupled} (pay attention to $\mathit{Step\,1}$). For Theorem~\ref{finalthm} we replace \eqref{kappaeps} by: 
	\begin{equation}\label{ext}
	\begin{gathered}
	\mbox{for every }(t,x)\in[0,+\infty)\times[\ell_0,+\infty)\text{ there exists }\varepsilon=\varepsilon(t,x)>0\\
	\text{such that } \kappa\in C^{0,1}([t,t+\varepsilon]\times[x,x+\varepsilon]),
	\end{gathered}		
	\end{equation}
	and we perform a similar proof: in order to start the machinery that leads to the existence of a unique solution to the coupled problem we only need to introduce a virtual toughness $\tilde\kappa$ for which we can apply Proposition~\ref{contractioncoupled}; such a $\tilde\kappa$ is obtained by extending $\kappa$ outside $[0,\varepsilon]\times[\ell_0,\ell_0+\varepsilon]$ (where $\varepsilon=\varepsilon(0,\ell_0)$) in a Lipschitz way and then truncating this extension between two suitable values. 
\end{rmk}
\begin{rmk}[\textbf{Lack of uniqueness and of existence}]\label{failureuniq}
	We want to remark that the right Lipschitzianity of the toughness $\kappa$ is crucial for the validity of Theorem~\ref{finalthm}, at least in the undamped case $\nu=0$. Indeed, removing that assumption, the following example shows how the coupled problem can have more than one (actually infinitely many) solution:\par 
	fix $\ell_0>0$ and let $\nu=0$; pick $u_0\equiv0$ and $u_1\equiv1$ in $[0,\ell_0]$, $w\equiv0$ in $[0,+\infty)$ and consider $\kappa(x)=\displaystyle\frac 12\max\left\{\frac{1-\sqrt{x-\ell_0}}{1+\sqrt{x-\ell_0}},\frac 12\right\}$ for every $x\ge\ell_0$. If the time $T$ is small enough the equation for $\ell$ in \eqref{coupledv} can be written in the following way:
	\begin{equation}\label{counterex}
	\begin{cases}
	\dot{\ell}(t)=\displaystyle\frac{1-2\kappa(\ell(t))}{1+2\kappa(\ell(t))}=\sqrt{\ell(t)-\ell_0}, &\mbox{for a.e. }t\in[0,T],\\
	\ell(0)=\ell_0.
	\end{cases}		
	\end{equation}
	It is well known that Cauchy problem \eqref{counterex} admits infinitely many solutions, for instance two of them are $\ell(t)=\ell_0$ and $\ell(t)=\frac{t^2}{4}+\ell_0$; so coupled problem \eqref{coupledv} admits infinitely many solutions as well.\par
	If instead $\kappa$ is neither right continuous, we can have no solutions to the coupled problem: under the previous assumptions consider $\kappa(x)=1/6$ if $x=\ell_0$ and $\kappa(x)=1/2$ otherwise, then (for $T$ small enough) the equation for $\ell$ reads as
	\begin{equation}\label{impossible}
		\dot{\ell}(t)=\begin{cases}
		1/2,&\text{if }\ell(t)=\ell_0,\\
		0,&\text{if }\ell(t)>\ell_0.
		\end{cases}\quad\quad\mbox{for a.e. }t\in[0,T].
	\end{equation}
	Since there are no Lipschitz solutions of \eqref{impossible} satisfying $\ell(0)=\ell_0$ we get that the coupled problem possesses no solutions as well.\par 
	This second example can be also adapted to the case of a piecewise constant and left continuous toughness, choosing properly the initial data $u_0$ and $u_1$.
\end{rmk}
\begin{rmk}[\textbf{Adding a forcing term}]\label{external}
	Following the same presentation of the paper one can also cover the case in which in the model an external force $f$ is present, namely when the equation for the vertical displacement $u$ is
	\begin{equation*}
		u_{tt}(t,x)-u_{xx}(t,x)+\nu u_t(t,x)=f(t,x), \quad t > 0 \,,\, 0<x<\ell(t).
	\end{equation*}
	For the forcing term $f$ we require
	\begin{equation}\label{effe}
		f\in L^2_{\textnormal{loc}}((0,+\infty)^2)\quad\text{such that}\quad f\in L^2((0,T)^2)\quad\text{for every }T>0,
	\end{equation}
	and we introduce the function $g(t,x):=e^{\nu t/2}f(t,x)$, so that $v(t,x)=e^{\nu t/2}u(t,x)$ solves
	\begin{equation*}
		v_{tt}(t,x)-v_{xx}(t,x)-\frac{\nu^2}{4} v(t,x)=g(t,x), \quad t > 0 \,,\, 0<x<\ell(t).
	\end{equation*}
	By Duhamel's principle the representation formula for $v$ now takes the form
	\begin{equation*}
			v(t,x)=A(t,x)+\frac{\nu^2}{8}\iint_{R(t,x)}v(\tau,\sigma)\d\sigma \d\tau+\frac 12\iint_{R(t,x)}g(\tau,\sigma)\d\sigma \d\tau, \quad \mbox{for a.e. }(t,x)\in\Omega',
	\end{equation*}
	and so we can repeat the proofs of Proposition~\ref{contractionCl} and Theorem~\ref{globalsolutionprescr}.\par 
	For the energetic analysis performed in Section~\ref{sec2} we also have to consider the work done by the external forces, namely $\mathcal{F}(t):=\displaystyle\int_{0}^{t}\int_{0}^{\ell(\tau)}f(\tau,\sigma)u_t(\tau,\sigma)\d\sigma\d\tau$; if we take into account the total energy, which now possesses an additional term, i.e. $\mathcal{T}(t)=\mc E(t)+\mc A(t)-\mc F(t)$, then Proposition~\ref{energyderivative} holds true modifying formula \eqref{totderv} (and analogously \eqref{totderu}) to
	\begin{equation*}
		\begin{aligned}
		\dot{\mathcal{T}}(t)\!=\!&-\!\frac{\dot\ell(t)}{2}\frac{1{-}\dot\ell(t)}{1{+}\dot\ell(t)}e^{-\nu t}\!\left[\dot v_0(\ell(t){-}t){-}v_1(\ell(t){-}t){-}\frac{\nu^2}{4}\!\!\int_{0}^{t}\!\!v(\tau,\tau{-}t{+}\ell(t))\d\tau{-}\!\!\int_{0}^{t}\!\!g(\tau,\tau{-}t{+}\ell(t))\d\tau\right]^2\\
		&+\dot w(t)\left[\dot w(t)+\frac{\nu}{2}w(t)-e^{-\frac{\nu t}{2}}\left(\dot v_0(t)+v_1(t)+\frac{\nu^2}{4}\int_{0}^{t}v(\tau,t{-}\tau)\d\tau+\int_{0}^{t}g(\tau,t{-}\tau)\d\tau\right)\right].
		\end{aligned}
	\end{equation*}
	We can also repeat the proof of Theorem~\ref{DERR}, obtaining that for a.e. $t\in\left[0,\frac{\ell_0}{2}\right]$ the dynamic energy release rate can be expressed as
	\begin{equation*}
	G_\alpha(t) = \frac{1}{2}\frac{1-\alpha}{1+\alpha} e^{-\nu t}\left[\dot{v}_0(\ell(t){-}t){-}v_1(\ell(t){-}t){-}\frac{\nu^2}{4}\int_{0}^{t}\!\!v(\tau,\tau{-}t{+}\ell(t))\d\tau{-}\int_{0}^{t}\!\!g(\tau,\tau{-}t{+}\ell(t))\d\tau\right]^2.
	\end{equation*}
	
	Always assuming \eqref{effe} we recover Lemma~\ref{itself} and Lemma~\ref{equicontinuity}, while for Proposition~\ref{contractioncoupled}, and hence for Theorem~\ref{finalthm}, we need to require
	\begin{equation}\label{effe2}
	 f\in L^\infty_{\textnormal{loc}}((0,+\infty)^2)\quad\text{such that}\quad f\in L^\infty((0,T)^2)\quad\text{for every }T>0;
	\end{equation}
	thanks to \eqref{effe2} we can perform their proofs replacing operator \eqref{Psi} by
	\begin{subequations}
		\begin{equation*}
		\Psi_1(v,\lambda)(t,x)=\left(A_{\ell_\lambda}(t,x)+\frac{\nu^2}{8} \iint_{R_{\ell_\lambda}(t,x)}\!\!\!\!v(\tau,\sigma)\d\sigma\d\tau+\frac12\iint_{R_{\ell_\lambda}(t,x)}\!\!\!\!g(\tau,\sigma)\d\sigma\d\tau\right)\chi_{Q_{\ell_\lambda}}(t,x),
		\end{equation*}
		\begin{equation*}
		\Psi_2(v,\lambda)(y)\!=\!\frac 12 \!\int_{-\ell_0}^{y}\!\!\left(\!\!1{+}\max\!\!\left\{\!\!\frac{\left[\dot{v}_0({-}s){-}v_1({-}s){-}\frac{\nu^2}{4}\!\int_{0}^{\lambda(s)}\!\!v(\tau,\tau{-}s)\d\tau{-}\!\int_{0}^{\lambda(s)}\!\!g(\tau,\tau{-}s)\d\tau\right]^2}{2e^{\nu\lambda(s)}\kappa(\lambda(s){-}s)},1\!\right\}\!\!\right)\!\!\!\d s,
		\end{equation*}
	\end{subequations}
		and arguing in the same way.\par 
		We point out that condition \eqref{effe2} is crucial for the validity of Theorem~\ref{finalthm}, as the following example shows: fix $\ell_0>0$ and let $\nu=0$; pick $u_0\equiv0$ in $[0,\ell_0]$, $w\equiv0$ in $[0,+\infty)$, $\kappa\equiv 1/2$ in $[\ell_0,+\infty)$ and consider $u_1(x)=\sqrt{2(\ell_0-x)^\frac 23+1}$ and $f(t,x)=\displaystyle\frac{2}{3(x-\ell_0)^\frac 13 \sqrt{2(x-\ell_0)^\frac 23+1}}$. Notice that $f$ satisfies \eqref{effe} but not \eqref{effe2} and that $f(t,x)=\frac{\d}{\d x}\sqrt{2(x-\ell_0)^\frac 23+1}$. 
		With these data, if $Y>0$ is small enough, the equation for $\lambda$ becomes
		\begin{equation*}
			\begin{cases}
			\dot{\lambda}(y)=1+(\lambda(y)-y-\ell_0)^\frac 23&\text{for a.e. }y\in[-\ell_0,-\ell_0+Y],\\
			\lambda(-\ell_0)=0,
			\end{cases}
		\end{equation*}
		and so, as in the first example of Remark~\ref{failureuniq}, we lose uniqueness of solutions to the coupled problem.
\end{rmk}

We conclude Section~\ref{sec4} proving Proposition~\ref{contractioncoupled}:
	\begin{proof}[Proof of Proposition~\ref{contractioncoupled}]
	During the proof the symbol $C$ is used to denote a constant, which may change from line to line, that does not depend on the value of $Y$.\par 
	By Lemma~\ref{itself} and by the definition of $\mc B_1$ we know that $\Psi$ maps $\mc B$ into itself (for suitable small $Y$), so we only need to show that there exists $Y\in (0,\ell_0)$ for which $\Psi$ is a contraction with respect to the distance $\d$ defined in \eqref{distance}.\medskip
	
	$\!\!\!\!\!\!\mathit{Step\,1.}\:\: \mathit{Lipschitz\:estimates\:on\:\Psi_2.}$\medskip\par
	Fix $(v^1,\lambda^1)$, $(v^2,\lambda^2)\in\mc B$; let us introduce for a.e. $y\in I_Y$ the function $j(y):=|\dot{v}_0(-y)|+|v_1(-y)|+1$ and notice that $j$ is in $L^2(-\ell_0,0)$. For the sake of clarity we also define for $i=1,2$ $\displaystyle\rho_{v^i,\lambda^i}(y):=\dot v_0({-}y)-v_1({-}y)-\frac{\nu^2}{4}\int_{0}^{\lambda^i(y)}\!\!\!\!\!v^i(\tau,\tau{-}y)\d\tau$ and we observe that $|\rho_{v^i,\lambda^i}(y)|\le C j(y)$ for a.e. $y\in I_Y$; then we compute:
	\begingroup
	\allowdisplaybreaks
	\begin{align*}
		&\quad\Vert\Psi_2(v^1,\lambda^1)-\Psi_2(v^2,\lambda^2)\Vert_{C^0(I_Y)}\le\frac 12\int_{{-}\ell_0}^{{-}\ell_0{+}Y}|\Lambda_{v^1,\lambda^1}(s)-\Lambda_{v^2,\lambda^2}(s)|\d s\\
		&\le\frac 12\int_{-\ell_0}^{-\ell_0+Y}\left|\frac{e^{\nu\lambda^2(s)}\kappa(\lambda^2(s){-}s)\left(\rho_{v^1,\lambda^1}(s)\right)^2-e^{\nu\lambda^1(s)}\kappa(\lambda^1(s){-}s)\left(\rho_{v^2,\lambda^2}(s)\right)^2}{2e^{\nu(\lambda^1(s){+}\lambda^2(s))}\kappa(\lambda^1(s){-}s)\kappa(\lambda^2(s){-}s)}\right|\d s\\
		&\le C\int_{-\ell_0}^{-\ell_0+Y}e^{\nu\lambda^2(s)}\kappa(\lambda^2(s){-}s)\left|\left(\rho_{v^1,\lambda^1}(s)\right)^2-\left(\rho_{v^2,\lambda^2}(s)\right)^2\right|\d s\\
		&\quad+C\int_{-\ell_0}^{-\ell_0+Y}\left(\rho_{v^2,\lambda^2}(s)\right)^2\left|e^{\nu\lambda^1(s)}\kappa(\lambda^1(s){-}s)-e^{\nu\lambda^2(s)}\kappa(\lambda^2(s){-}s)\right|\d s\\
		&\le C\left[\int_{-\ell_0}^{-\ell_0+Y}\!\!\!\!\!\!\!\!\!j(s)\left|\int_{0}^{\lambda^1(s)}v^1(\tau,\tau{-}s)\d\tau-\int_{0}^{\lambda^2(s)}v^2(\tau,\tau{-}s)\d\tau\right|\d s+\int_{-\ell_0}^{-\ell_0+Y}\!\!\!\!\!\!\!\!\!j^2(s)|\lambda^2(s)-\lambda^1(s)|\d s\right]\\
		&\le C\left[\int_{-\ell_0}^{-\ell_0+Y}\!\!\!\!\!\!\!\!\!\!\!\!\!\!\!\!j(s)\int_{0}^{T}\!\!\!\!\!\!|v^1{-}v^2|(\tau,\tau{-}s)\d\tau\d s\!+\!\!\int_{-\ell_0}^{-\ell_0+Y}\!\!\!\!\!\!\!\!\!\!\!\!\!j(s)\left|\int_{\lambda^2(s)}^{\lambda^1(s)}\!\!\!\!\!\!\!\!\!\!\!\!\!v^2(\tau,\tau{-}s)\d\tau\right|\!\!\d s\!+\!\!\left(\int_{-\ell_0}^{-\ell_0+Y}\!\!\!\!\!\!\!\!\!\!\!\!\!\!\!\!\!j^2(s)\d s\right)\!\Vert\lambda^2{-}\lambda^1\Vert_{C^0(I_Y)}\right]\\
		&\le C\left[\left(\int_{-\ell_0}^{-\ell_0+Y}\!\!\!\!\!\!\!\!\!\!\!\!\!\!\!\!\!j^2(s)\d s\right)^{\!\!\!\frac 12}\!\!\Vert v^1{-}v^2\Vert_{L^2(Q)}\!+\!\left(\int_{-\ell_0}^{-\ell_0+Y}\!\!\!\!\!\!\!\!\!\!\!\!\!j(s)\d s\right)\Vert\lambda^2{-}\lambda^1\Vert_{C^0(I_Y)}\!+\!\left(\int_{-\ell_0}^{-\ell_0+Y}\!\!\!\!\!\!\!\!\!\!\!\!\!\!\!\!\!j^2(s)\d s\right)\!\Vert\lambda^2{-}\lambda^1\Vert_{C^0(I_Y)}\right]\\
		&\le C\left[\left(\int_{-\ell_0}^{-\ell_0+Y}\!\!\!\!\!\!\!\!\!j^2(s)\d s\right)^{\frac 12}+\int_{-\ell_0}^{-\ell_0+Y}\!\!\!\!\!\!\!\!\!\left(j(s)+j^2(s))\right)\d s\right]\d\left((v^1,\lambda^1),(v^2,\lambda^2)\right).
	\end{align*}
	\endgroup
	Since $j$ belongs to $L^2(-\ell_0,0)$ we deduce that choosing $Y$ small enough we get:
	\begin{equation}\label{firstestimate}
		\Vert\Psi_2(v^1,\lambda^1)-\Psi_2(v^2,\lambda^2)\Vert_{C^0(I_Y)}\le\frac 12\d\left((v^1,\lambda^1),(v^2,\lambda^2)\right).
	\end{equation}
	
	\medskip
	
	$\!\!\!\!\!\!\mathit{Step\,2.}\:\: \mathit{Lipschitz\:estimates\:on\:\Psi_1.}$\medskip\par
	Fix $(v^1,\lambda^1)$, $(v^2,\lambda^2)\in \mc B$ and let us define for the sake of clarity, as in the proof of Lemma~\ref{equicontinuity}, the function $\displaystyle H_{v,\lambda}(t,x):= \left(\iint_{R_\ell(t,x)}v(\tau,\sigma)\d\sigma\d\tau\right)\chi_{Q_\ell}(t,x)$, so that
	\begin{equation*}
	\Vert\Psi_1(v^1,\lambda^1)-\Psi_1(v^2,\lambda^2) \Vert_{L^2(Q)}\!\le\!\Vert A_{\ell^1}\chi_{Q_{\ell^1}}-A_{\ell^2}\chi_{Q_{\ell^2}} \Vert_{L^2(Q)}+\frac{\nu^2}{8}\Vert H_{v^1,\lambda^1}-H_{v^2,\lambda^2} \Vert_{L^2(Q)}.
	\end{equation*}
	We estimate the two norms separately. First of all we rewrite the square of the first term as
	\begin{equation}\label{splitA}
	\begin{aligned}
	\Vert A_{\ell^1}\chi_{Q_{\ell^1}}-A_{\ell^2}\chi_{Q_{\ell^2}} \Vert^2_{L^2(Q)}&=\iint_{Q_{\ell^1}\setminus Q_{\ell^2}}|A_{\ell^1}(t,x)|^2\d x\d t+\iint_{Q_{\ell^2}\setminus Q_{\ell^1}}|A_{\ell^2}(t,x)|^2\d x\d t\\
	&\quad+\iint_{Q_{\ell^1}\cap Q_{\ell^2}}|A_{\ell^1}(t,x)-A_{\ell^2}(t,x)|^2\d x\d t,
	\end{aligned}
	\end{equation}
	and we notice that for every $s\in[\ell_0,\min\{(\omega^1)^{-1}({-}\ell_0{+}Y),(\omega^2)^{-1}({-}\ell_0{+}Y)\} ]$ it holds:
		\begin{equation*}
		\begin{aligned}
		|\omega^1(s)-\omega^2(s)|&=|\lambda^1(\omega^1(s))-\lambda^2(\omega^2(s))-\ell^1(\lambda^1(\omega^1(s)))+\ell^2(\lambda^2(\omega^2(s)))|\\
		&=2|\ell^1(\lambda^1(\omega^1(s)))-\ell^2(\lambda^2(\omega^2(s)))|\\
		&\le2|\ell^1(\lambda^1(\omega^1(s)))-\ell^2(\lambda^2(\omega^1(s)))|\\
		&=2|\lambda^1(\omega^1(s))-\lambda^2(\omega^1(s))|\\
		&\le 2\Vert \lambda^1-\lambda^2\Vert_{C^0(I_Y)}.
		\end{aligned}
		\end{equation*}
		\begin{figure}
			\centering
			\includegraphics[scale=.5]{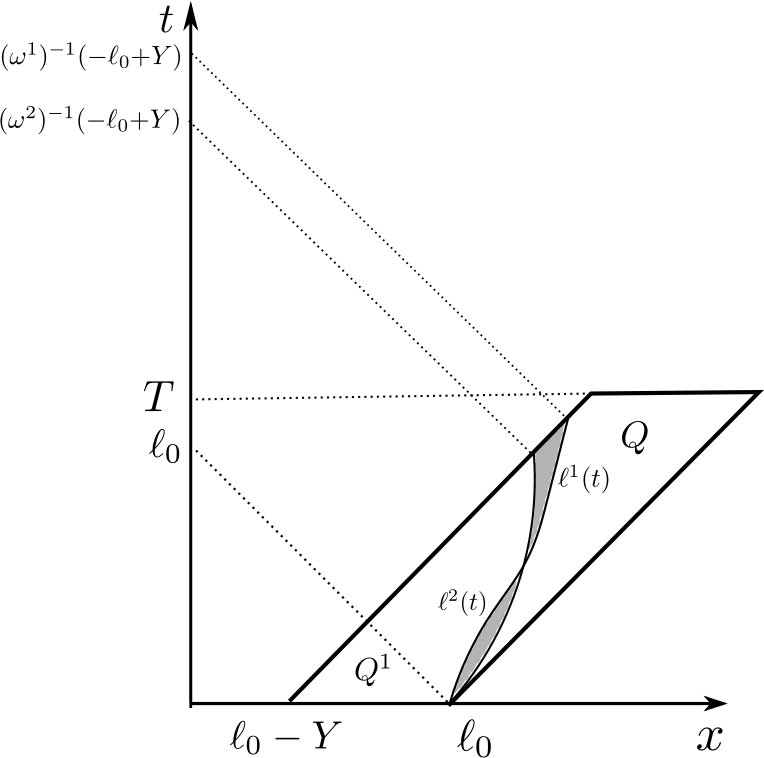}
			\caption{The set $Q^1$ and, in grey, the symmetric difference $Q_{\ell^1}\triangle Q_{\ell^2}$.}\label{Q2}
		\end{figure}\noindent
		This in particular implies (we define $Q^3_{\ell^i}:=Q_{\ell^i}\cap\overline{(\Omega_{\ell^i})'_3}$):
		\begin{subequations}
		\begin{equation}\label{supa}
			|\omega^1(x+t)-\omega^2(x+t)|\le 2\Vert \lambda^1-\lambda^2\Vert_{C^0(I_Y)},\quad\text{if }(t,x)\in Q^3_{\ell^1}\cap Q^3_{\ell^2},
		\end{equation}
		\begin{equation}\label{supb}
			|(t-x)-\omega^1(x+t)|\le 2\Vert \lambda^1-\lambda^2\Vert_{C^0(I_Y)},\quad\text{if }(t,x)\in Q_{\ell^1}\setminus Q_{\ell^2},
		\end{equation}
	\end{subequations}
	and the same holds interchanging the role of $1$ and $2$ in \eqref{supb}.\\  
	Moreover the measure of the symmetric difference of $Q_{\ell^1}$ and $Q_{\ell^2}$ can be estimated as
		\begin{equation}\label{diff}
			|Q_{\ell^1}\triangle Q_{\ell^2}|=\int_{{-}\ell_0}^{{-}\ell_0+Y}|\lambda^1(s)-\lambda^2(s)|\d s\le Y\Vert\lambda^1-\lambda^2\Vert_{C^0(I_Y)}.
		\end{equation}
	For $(t,x)\in Q_{\ell^1}\setminus Q_{\ell^2} $, exploiting the explicit form of $A$ given by \eqref{homsol} and using \eqref{supb}, we deduce:
	\begin{align*}
		|A_{\ell^1}(t,x)|^2&=\frac 14\left|\int_{x-t}^{-\omega^1(x+t)}(v_1(s)-\dot{v}_0(s))\d s\right|^2\le\frac 14 |(t-x)-\omega^1(x+t)|\int_{0}^{\ell_0}|v_1(s)-\dot{v}_0(s)|^2\d s\\
		&\le C\Vert\lambda^1-\lambda^2\Vert_{C^0(I_Y)}.
	\end{align*}
	So, by \eqref{diff}, we get:
	\begin{equation}\label{ablabla}
	\begin{aligned}
	\iint_{Q_{\ell^1}\setminus Q_{\ell^2}}\!\!\!\!\!\!\!\!\!\!\!\!\!\!\!\!|A_{\ell^1}(t,x)|^2\d x\d t+\iint_{Q_{\ell^2}\setminus Q_{\ell^1}}\!\!\!\!\!\!\!\!\!\!\!\!\!\!\!\!|A_{\ell^2}(t,x)|^2\d x\d t&\le C \Vert\lambda^1-\lambda^2\Vert_{C^0(I_Y)}|Q_{\ell^1}\triangle Q_{\ell^2}|\\
	&\le CY\Vert\lambda^1-\lambda^2\Vert_{C^0(I_Y)}^2.
	\end{aligned}
	\end{equation}	
	To estimate the term in the second line in \eqref{splitA} we firstly notice that $A_{\ell^1}-A_{\ell^2}$ vanishes on $Q^1:=Q\cap\overline{\Omega'_1}$ (we remark that $Q^1=\overline{Q_{\ell^1}\setminus Q^3_{\ell^1}}=\overline{Q_{\ell^2}\setminus Q^3_{\ell^2}}$ does not depend on $\ell^i$, see also Figure~\ref{Q2}), while for $(t,x)\in Q^3_{\ell^1}\cap Q^3_{\ell^2}$, using \eqref{supa}, we have:
	\begin{align*}
		|A_{\ell^1}(t,x)-A_{\ell^2}(t,x)|^2&=\frac 14\left|\int_{-\omega^2(x+t)}^{-\omega^1(x+t)}(v_1(s)-\dot{v}_0(s))\d s\right|^2\\
		&\le \frac{|\omega^1(x+t)-\omega^2(x+t)|}{4}\left|\int_{-\omega^2(x+t)}^{-\omega^1(x+t)}|v_1(s)-\dot{v}_0(s)|^2\d s\right|\\
		&\le \frac 12 \Vert\lambda^1-\lambda^2\Vert_{C^0(I_Y)}\left|\int_{-\omega^2(x+t)}^{-\omega^1(x+t)}|v_1(s)-\dot{v}_0(s)|^2\d s\right|.
	\end{align*}
	So we deduce:
	\begin{align*}
		&\iint_{Q_{\ell^1}\cap Q_{\ell^2}}\!\!\!\!\!\!\!\!\!\!\!\!\!\!\!\!|A_{\ell^1}(t,x)-A_{\ell^2}(t,x)|^2\!\d x\d t\le \frac 12 \Vert\lambda^1-\lambda^2\Vert_{C^0(I_Y)}\!\!\iint_{Q^3_{\ell^1}\cap Q^3_{\ell^2}}\left|\int_{-\omega^2(x+t)}^{-\omega^1(x+t)}\!\!\!\!\!|v_1(s)-\dot{v}_0(s)|^2\d s\right|\!\!\d x\d t\\
		&=\frac 12 \Vert\lambda^1-\lambda^2\Vert_{C^0(I_Y)}\int_{\ell_0}^{m(Y)}\int_{(\psi^1)^{-1}(b)\vee(\psi^2)^{-1}(b)}^{\frac{b+Y-\ell_0}{2}}\left|\int_{-\omega^2(b)}^{-\omega^1(b)}\!\!\!\!\!|v_1(s)-\dot{v}_0(s)|^2\d s\right|\d a\d b=:(\dagger),
	\end{align*}
	where we performed the change of variables $\begin{cases}
	a=t,\\
	b=x+t
	\end{cases}$, denoted by $m(Y)$ the minimum between  $(\omega^1)^{-1}({-}\ell_0{+}Y)$ and $(\omega^2)^{-1}({-}\ell_0{+}Y)$ and used the symbol $\vee$ to denote the maximum between two numbers. We continue the estimate using Fubini's Theorem:
	\begingroup
	\allowdisplaybreaks
	\begin{align*}
		(\dagger)&\le\frac 12 \Vert\lambda^1-\lambda^2\Vert_{C^0(I_Y)}\int_{\ell_0}^{m(Y)}Y\left|\int_{-\omega^2(b)}^{-\omega^1(b)}\!\!\!\!\!|v_1(s)-\dot{v}_0(s)|^2\d s\right|\d b\\
		&=\frac{Y}{2} \Vert\lambda^1-\lambda^2\Vert_{C^0(I_Y)}\int_{\ell_0{-}Y}^{\ell_0}\!\!\!\!\!|v_1(s)-\dot{v}_0(s)|^2\left|\int_{(\omega^2)^{-1}(-s)}^{(\omega^1)^{-1}(-s)}\d b\right|\d s\\
		&=\frac{Y}{2} \Vert\lambda^1-\lambda^2\Vert_{C^0(I_Y)}\int_{\ell_0{-}Y}^{\ell_0}\!\!\!\!\!|v_1(s)-\dot{v}_0(s)|^2|\lambda^1({-}s){-}\lambda^2({-}s){+}\ell^1(\lambda^1({-}s)){-}\ell^2(\lambda^2({-}s))|\d s\\
		&=Y\Vert\lambda^1-\lambda^2\Vert_{C^0(I_Y)}\int_{\ell_0{-}Y}^{\ell_0}\!\!\!\!\!|v_1(s)-\dot{v}_0(s)|^2|\lambda^1({-}s){-}\lambda^2({-}s)|\d s\\
		&\le CY\Vert\lambda^1-\lambda^2\Vert_{C^0(I_Y)}^2.
	\end{align*}
	\endgroup
	Combining the previous estimate with \eqref{ablabla} and \eqref{splitA} we get:
	\begin{equation}\label{estimateA}
		\Vert A_{\ell^1}\chi_{Q_{\ell^1}}-A_{\ell^2}\chi_{Q_{\ell^2}} \Vert_{L^2(Q)}\le C\sqrt{Y}\Vert\lambda^1-\lambda^2\Vert_{C^0(I_Y)}
	\end{equation}
	Concerning $\Vert H_{v^1,\lambda^1}-H_{v^2,\lambda^2} \Vert_{L^2(Q)}$ we split its square as in \eqref{splitA}:
	\begin{equation}\label{splitH}
	\begin{aligned}
	\Vert  H_{v^1,\lambda^1}-H_{v^2,\lambda^2} \Vert^2_{L^2(Q)}&=\iint_{Q_{\ell^1}\setminus Q_{\ell^2}}\left|\iint_{R_{\ell^1}(t,x)}\!\!\!\!\!\!\!\!\!\!\!\!\!\!v^1(\tau,\sigma)\d\sigma\d\tau\right|^2\d x\d t\\
	&\quad+\iint_{Q_{\ell^2}\setminus Q_{\ell^1}}\left|\iint_{R_{\ell^2}(t,x)}\!\!\!\!\!\!\!\!\!\!\!\!\!\!v^2(\tau,\sigma)\d\sigma\d\tau\right|^2\d x\d t\\
	&\quad+\iint_{Q_{\ell^1}\cap Q_{\ell^2}}\left|\iint_{R_{\ell^1}(t,x)}\!\!\!\!\!\!\!\!\!\!\!\!\!\!v^1(\tau,\sigma)\d\sigma\d\tau-\iint_{R_{\ell^2}(t,x)}\!\!\!\!\!\!\!\!\!\!\!\!\!\!v^2(\tau,\sigma)\d\sigma\d\tau\right|^2\d x\d t,
	\end{aligned}
	\end{equation}
	and we denote by $\mc I$, $\mc{II}$ and $\mc{III}$ the expressions in the first, second and third line of \eqref{splitH}, respectively. Exploiting \eqref{supb} and \eqref{diff} we get:
	\begingroup
	\allowdisplaybreaks
	\begin{align*}
		\mc I+\mc{II}&\le\iint_{Q_{\ell^1}\setminus Q_{\ell^2}}M^2|R_{\ell^1}(t,x)|^2\d x\d t+\iint_{Q_{\ell^2}\setminus Q_{\ell^1}}M^2|R_{\ell^2}(t,x)|^2\d x\d t\\&\le\iint_{Q_{\ell^1}\setminus Q_{\ell^2}}M^2T^2|(t{-}x)-\omega^1(x{+}t)|^2\d x\d t+\iint_{Q_{\ell^2}\setminus Q_{\ell^1}}M^2T^2|(t{-}x)-\omega^2(x{+}t)|^2\d x\d t\\
		&\le 4M^2T^2\Vert \lambda^1-\lambda^2\Vert^2_{C^0(I_Y)}|Q_{\ell^1}\triangle Q_{\ell^2}|\le 8M^2T^3Y\Vert \lambda^1-\lambda^2\Vert^2_{C^0(I_Y)},
	\end{align*}
	\endgroup
	while we estimate $\mc{III}$ using again \eqref{supa}:
	\begingroup
	\allowdisplaybreaks
	\begin{align*}
		\mc{III}&\le\iint_{Q^1}\left(\iint_{R(t,x)}|v^1-v^2|(\tau,\sigma)\d\sigma\d\tau\right)^2\d x\d t\\
		&\quad+\iint_{Q^3_{\ell^1}\cap Q^3_{\ell^2}}\left(\iint_{R_{\ell^1}(t,x)}\!\!\!\!\!\!\!\!\!\!\!\!\!\!|v^1-v^2|(\tau,\sigma)\d\sigma\d\tau+\iint_{R_{\ell^1}(t,x)\triangle R_{\ell^2}(t,x)}\!\!\!\!\!\!\!\!\!\!\!\!\!\!\!\!\!\!\!\!\!\!\!\!\!\!|v^2(\tau,\sigma)|\d\sigma\d\tau\right)^2\d x\d t\\
		&\le C\left[|Q|\Vert v^1-v^2\Vert^2_{L^2(Q)}+\iint_{Q^3_{\ell^1}\cap Q^3_{\ell^2}}\left(|Q|\Vert v^1-v^2\Vert^2_{L^2(Q)}+|R_{\ell^1}(t,x)\triangle R_{\ell^2}(t,x)|^2\right)\d x\d t\right]\\
		&\le C\left[|Q|\Vert v^1-v^2\Vert^2_{L^2(Q)}+\iint_{Q^3_{\ell^1}\cap Q^3_{\ell^2}}\left(\Vert v^1-v^2\Vert^2_{L^2(Q)}+|\omega^1(x+t)-\omega^2(x+t)|^2\right)\d x\d t\right]\\
		&\le C\left[|Q|\Vert v^1-v^2\Vert^2_{L^2(Q)}+|Q^3_{\ell^1}\cap Q^3_{\ell^2}|\left(\Vert v^1-v^2\Vert^2_{L^2(Q)}+\Vert\lambda^1-\lambda^2\Vert^2_{C^0(I_Y)}\right)\right]\\
		&\le CY\d\left((v^1,\lambda^1),(v^2,\lambda^2)\right)^2.
	\end{align*}
	\endgroup	
	So we infer:
	\begin{align}\label{estimateH}
		\Vert  H_{v^1,\lambda^1}-H_{v^2,\lambda^2} \Vert^2_{L^2(Q)}=\mc I+\mc{II}+\mc{III}\le CY\d\left((v^1,\lambda^1),(v^2,\lambda^2)\right)^2.
	\end{align}	
	Using \eqref{estimateA} and \eqref{estimateH} and choosing $Y$ small enough we finally deduce:
	\begin{equation}\label{secondestimate}
		\Vert\Psi_1(v^1,\lambda^1)-\Psi_1(v^2,\lambda^2) \Vert_{L^2(Q)}\le \frac 12\d\left((v^1,\lambda^1),(v^2,\lambda^2)\right).
	\end{equation}
	
	\medskip
	
	$\!\!\!\!\!\!\mathit{Step\,3.}\:\: \mathit{\Psi\colon\mc B\to \mc B\:is\:a\:contraction.}$\medskip\par
	Combining estimates \eqref{firstestimate} and \eqref{secondestimate} we obtain:
	\begin{align*}
		\d\left(\Psi(v^1,\lambda^1),\Psi(v^2,\lambda^2)\right)&=\max\{\Vert\Psi_1(v^1,\lambda^1)-\Psi_1(v^2,\lambda^2) \Vert_{L^2(Q)}, \Vert \Psi_2(v^1,\lambda^1)-\Psi_2(v^2,\lambda^2)\Vert_{C^0(I_Y)}\}\\
		&\le\frac 12\d\left((v^1,\lambda^1),(v^2,\lambda^2)\right).
	\end{align*}
	This shows that for a suitable choice of $Y\in(0,\ell_0)$ the operator $\Psi$ is a contraction in $(\mc B,\d)$, and we conclude. 	
\end{proof}
	\par
\appendix
\section{Chain rule and Leibniz differentiation rule}
In this Appendix we gather some results about the Chain rule and the Leibniz differentiation rule under low regularity assumptions. These results have been used throughout the paper and they are of some interest on their own.\par
For the sake of brevity we assume that in all the statements the function $\varphi$ is nondecreasing (or strictly increasing), although they are still valid if $\varphi$ is nonincreasing (or strictly decreasing), with little changes in the proofs. 
	\begin{lemma}[\textbf{Change of variables formula}]\label{areaformula}
	Let $\varphi\in C^{0,1}([a,b])$ be nondecreasing. Then for every nonnegative and measurable function $g$ on $[\varphi(a),\varphi(b)]$ (and hence for every $g\in L^1(\varphi(a),\varphi(b)$) it holds
	\begin{equation}\label{formulaa}
	\int_{\varphi(a)}^{\varphi(b)}g(y)\d y=\int_{a}^{b}g(\varphi(s))\dot{\varphi}(s)\d s.
	\end{equation}
\end{lemma}
\begin{rmk}\label{precisini}
	In general the expression $g(\varphi(s))\dot{\varphi}(s)$ in \eqref{formulaa} has to be meant replacing $g$ by a Borel function $\tilde{g}$ equal to $g$ a.e. in $[\varphi(a),\varphi(b)]$ and finite everywhere (if $g$ is finite a.e.); in the particular case in which $\dot{\varphi}(t)>0$ for a.e. $t\in[a,b]$ that expression is meaningful without modifications on sets of measure zero (see Corollary \ref{nullsets}).
\end{rmk}
\begin{proof}[Proof of Lemma~\ref{areaformula}]
	If $\varphi$ is strictly increasing, hence injective, the result is well known. If not, by the Area Formula for Lipschitz maps (see~\cite{Krantz}, Corollary 5.1.13) we have
	\begin{equation}\label{formula}
	\int_{\varphi(a)}^{\varphi(b)}g(y)\#\varphi^{-1}(\{y\})\d y=\int_{a}^{b}g(\varphi(s))\dot{\varphi}(s)\d s.
	\end{equation}
	We conclude if we prove that $\#\varphi^{-1}(\{y\})=1$ for a.e. $y\in[\varphi(a),\varphi(b)]$.\par
	Since $\varphi$ is nondecreasing and continuous, for every $y\in[\varphi(a),\varphi(b)]$ the set $\varphi^{-1}(\{y\})$ can be either a singleton either a closed interval, so $\#\varphi^{-1}(\{y\})\in\{1,+\infty\}$. Taking $g\equiv 1$ in \eqref{formula} we deduce
	\begin{equation*}
	+\infty>\varphi(b)-\varphi(a)=\int_{a}^{b}\dot{\varphi}(s)\d s=	\int_{\varphi(a)}^{\varphi(b)}\#\varphi^{-1}(\{y\})\d y.
	\end{equation*}
	This yields $\#\varphi^{-1}(\{y\})<+\infty$ for a.e. $y\in[\varphi(a),\varphi(b)]$ and so necessarily $\#\varphi^{-1}(\{y\})=1$ a.e..\par
	As an alternative proof we notice that the set $\{y\in[\varphi(a),\varphi(b)]\mid\#\varphi^{-1}(\{y\})=+\infty\}$ is in bijection with a subset of rational numbers, so it is countable and hence of measure zero.
\end{proof}
\begin{rmk}\label{absolutecont}
	Formula \eqref{formulaa} still holds true only assuming that $\varphi$ is absolutely continuous on $[a,b]$ (and nondecreasing), see Theorem~7.26 in~\cite{Rud}. This ensures that every result in this Appendix is valid replacing the assumption $\varphi\in C^{0,1}([a,b])$ by $\varphi\in AC([a,b])$; indeed the reader can easily check that the only ingredient needed to carry out all the proofs is \eqref{formulaa}.
\end{rmk}
\begin{cor}\label{nullsets}
	Let $\varphi\in C^{0,1}([a,b])$ be nondecreasing and let $N\subset[\varphi(a),\varphi(b)]$ be a set of measure zero. Then the set $M=\{t\in\varphi^{-1}(N)\mid\dot{\varphi}(t)\text{ exists and }\dot{\varphi}(t)>0\}$ has measure zero as well. In particular, if $\dot{\varphi}(t)>0$ for a.e. $t\in[a,b]$, then $\varphi^{-1}$ maps sets of measure zero in sets of measure zero.
\end{cor}
\begin{proof}
	Let $N\subset[\varphi(a),\varphi(b)]$ be a set of measure zero; then by Lemma~\ref{areaformula}
	\begin{equation*}
		0=\int_{\varphi(a)}^{\varphi(b)}\chi_N(y)\d y=\int_{a}^{b}\chi_N(\varphi(s))\dot{\varphi}(s)\d s=\int_{\varphi^{-1}(N)}^{}\dot{\varphi}(s)\d s=\int_{M}^{}\dot{\varphi}(s)\d s.
	\end{equation*}
	Since by construction $\dot{\varphi}(t)>0$ for every $t\in M$, we deduce that the set $M$ has measure zero.
\end{proof}
\begin{cor}\label{absinv}
	Let $\varphi\in C^{0,1}([a,b])$ be a strictly increasing function such that $\dot{\varphi}(t)>0$ for a.e. $t\in[a,b]$. Then $\varphi^{-1}$ belongs to $AC([\varphi(a),\varphi(b)])$and $\displaystyle\frac{\d}{\d x}(\varphi^{-1})(x)=\frac{1}{\dot{\varphi}(\varphi^{-1}(x))}$ for a.e. $x\in[\varphi(a),\varphi(b)]$.
\end{cor}
\begin{proof}
	Firstly we notice that Lemma~\ref{areaformula} ensures that $\displaystyle\frac{1}{\dot{\varphi}\circ\varphi^{-1}}$ belongs to $L^1(\varphi(a),\varphi(b)$:
	\begin{equation*}
		\int_{\varphi(a)}^{\varphi(b)}\frac{1}{\dot{\varphi}(\varphi^{-1}(y))}\d y=\int_{a}^{b}\frac{1}{\dot{\varphi}(s)}\dot{\varphi}(s)\d s=b-a<+\infty.
	\end{equation*}
	Moreover for every $x\in[\varphi(a),\varphi(b)]$
	\begin{equation*}
		\varphi^{-1}(x)-\varphi^{-1}(\varphi(a))=\int_{a}^{\varphi^{-1}(x)}\d s=\int_{a}^{\varphi^{-1}(x)}\frac{\dot{\varphi}(s)}{\dot{\varphi}(s)}\d s=\int_{\varphi(a)}^{x}\frac{1}{\dot{\varphi}(\varphi^{-1}(y))}\d y,
	\end{equation*}
	so we conclude.
\end{proof}
\begin{lemma}[\textbf{Chain rule}]
	Let $\varphi\in C^{0,1}([a,b])$ be nondecreasing and let $\phi\in AC([\varphi(a),\varphi (b)])$. Then $\phi\circ\varphi$ belongs to $AC([a,b])$ and $\frac{\d}{\d t}(\phi\circ\varphi)(t)=\dot \phi(\varphi(t))\dot{\varphi}(t)$ for a.e. $t\in[a,b]$, where the right-hand side is meant as in Remark~\ref{precisini}. 
\end{lemma}
\begin{proof}
	Since $\phi\in AC([\varphi(a),\varphi (b)])$, Lemma~\ref{areaformula} ensures that $\dot \phi(\varphi(\cdot))\dot{\varphi}(\cdot)$ belongs to $L^1(a,b)$. Moreover  for every $t\in[a,b]$
	\begin{equation*}
	\phi(\varphi(t))-\phi(\varphi(a))=\int_{\varphi(a)}^{\varphi(t)}\dot \phi(y)\d y=\int_{a}^{t}\dot \phi(\varphi(s))\dot{\varphi}(s)\d s,
	\end{equation*}
	so we conclude.
\end{proof}
\begin{rmk}\label{sobolev}
	With a similar proof one can show that if $\phi\in W^{1,p}(\varphi(a),\varphi(b))$ for $p\in[1,+\infty]$, then $\phi\circ\varphi\in W^{1,p}(a,b)$ and the same formula for the derivative holds. In contrast with Remark~\ref{absolutecont}, for the validity of this fact we cannot replace $\varphi\in C^{0,1}([a,b])$ by $\varphi\in AC([a,b])$. 
\end{rmk}
\begin{thm}[\textbf{Leibniz differentiation rule}]\label{differentiation}
	Let $\varphi\in C^{0,1}([0,T])$ be nondecreasing and let $a\le\varphi(0)$. Consider the set $\Omega_T^\varphi:=\{(t,y)\mid 0\le t\le T,\,a\le y\le\varphi(t)\}$ and let $f\colon\Omega_T^\varphi\to\erre$ be a measurable function such that:
	\begin{itemize}
		\item[a)] for every $t\in[0,T]$ it holds $f(t,\cdot)\in L^1(a,\varphi(t))$,
		\item[b)] for a.e. $y\in[a,\varphi(T)]$ it holds $f(\cdot,y)\in AC(I_y)$, where $I_y=\{t\in[0,T]\mid y\le \varphi(t) \}$,
		\item[c)] the partial derivative $\displaystyle\frac{\de f}{\de t}(t,y):=\lim\limits_{h\to 0}\frac{f(t+h,y)-f(t,y)}{h}$ (which for a.e. $y\in[a,\varphi(T)]$ is well defined for a.e. $t\in I_y$) is summable in $\Omega_T^\varphi$.
	\end{itemize}
	Then the function $F(t):=\displaystyle\int_{a}^{\varphi(t)}f(t,y)\d y$ belongs to $AC([0,T])$ and moreover for a.e. $t\in[0,T]$
	\begin{equation}\label{leibniz}
	\dot F(t)=f(t,\varphi(t))\dot{\varphi}(t)+\int_{a}^{\varphi(t)}\frac{\de f}{\de t}(t,y)\d y.
	\end{equation}
\end{thm}
\begin{proof}
To conclude we need to prove two things :
	\begin{itemize}
		\item[1)] The right-hand side in \eqref{leibniz} belongs to $L^1(0,T)$.
		\item[2)] $\displaystyle F(t)=\int_{a}^{\varphi(T)}\!\!\!\!\!\!\!\!f(T,y)\d y-\int_{t}^{T}\!\!\!f(s,\varphi(s))\dot{\varphi}(s)\d s-\int_{t}^{T}\!\!\!\int_{a}^{\varphi(s)}\frac{\de f}{\de t}(s,y)\d y\d s$, for every $t\in[0,T]$.
	\end{itemize}	
	To prove 1) notice that the integral part in the formula belongs to $L^1(0,T)$ by c) and Fubini's Theorem. To ensure that also $f(\cdot,\varphi(\cdot))\dot{\varphi}(\cdot)\in L^1(0,T)$ we argue as follows:
	\begin{itemize}
		\item[-] By b) we know that for a.e. $y\in[a,\varphi(T)]$ it holds $\displaystyle f(t,y)=f(T,y)-\int_{t}^{T}\frac{\de f}{\de t}(s,y)\d s$ for every $t\in I_y$,
		\item[-] since $\varphi$ is continuous and nondecreasing we know that for a.e. $y\in[\varphi(0),\varphi(T)]$ there exists a unique element of $[0,T]$, denoted by $\varphi^{-1}(y)$, such that $\varphi(\varphi^{-1}(y))=y$ (see the proof of Lemma~\ref{areaformula}).
	\end{itemize}
	Hence $f(\varphi^{-1}(y),y)=f(T,y)-\displaystyle\int_{\varphi^{-1}(y)}^{T}\frac{\de f}{\de t}(s,y)\d s$ for a.e $y\in[\varphi(0),\varphi(T)]$, and so
	\begin{align*}
	\int_{\varphi(0)}^{\varphi(T)}|f(\varphi^{-1}(y),y)|\d y&\le\int_{\varphi(0)}^{\varphi(T)}|f(T,y)|\d y+\int_{\varphi(0)}^{\varphi(T)}\int_{\varphi^{-1}(y)}^{T}\left|\frac{\de f}{\de t}\right|(s,y)\d s\d y\\
	&\le\Vert f(T,\cdot)\Vert_{L^1(a,\varphi(T))}+\left\Vert \frac{\de f}{\de t}\right\Vert_{L^1(\Omega_T^\varphi)}<+\infty.
	\end{align*}
	Using Lemma~\ref{areaformula} and recalling Corollary \ref{nullsets} we deduce:
	\begin{equation*}
	+\infty>\int_{\varphi(0)}^{\varphi(T)}|f(\varphi^{-1}(y),y)|\d y=\int_{0}^{T}|f(s,\varphi(s))|\dot{\varphi}(s)\d s.
	\end{equation*}	
	Now we prove 2). Fix $t\in[0,T]$, then
	\begingroup
	\allowdisplaybreaks
	\begin{align*}
	F(t)&=\int_{a}^{\varphi(t)}f(t,y)\d y=\int_{a}^{\varphi(t)}f(T,y)\d y-\int_{a}^{\varphi(t)}\int_{t}^{T}
	\frac{\de f}{\de t}(s,y)\d s\d y\\
	&=\int_{a}^{\varphi(T)}f(T,y)\d y-\int_{\varphi(t)}^{\varphi(T)}f(T,y)\d y-\int_{t}^{T}\int_{a}^{\varphi(t)}
	\frac{\de f}{\de t}(s,y)\d y\d s\\
	&=\int_{a}^{\varphi(T)}\!\!\!f(T,y)\d y-\int_{t}^{T}\!\!\int_{a}^{\varphi(s)}\frac{\de f}{\de t}(s,y)\d y\d s-\int_{\varphi(t)}^{\varphi(T)}\!\!\!f(T,y)\d y+\int_{t}^{T}\!\!\int_{\varphi(t)}^{\varphi(s)}\frac{\de f}{\de t}(s,y)\d y\d s.
	\end{align*}
	\endgroup
	So we conclude if we prove $\displaystyle-\int_{\varphi(t)}^{\varphi(T)}\!\!\!f(T,y)\d y+\int_{t}^{T}\!\!\int_{\varphi(t)}^{\varphi(s)}\frac{\de f}{\de t}(s,y)\d y\d s=-\int_{t}^{T}\!\!\!f(s,\varphi(s))\dot{\varphi}(s)\d s$. This is true by the following computation:
	\begin{align*}
	&\quad\,\int_{t}^{T}\int_{\varphi(t)}^{\varphi(s)}\frac{\de f}{\de t}(s,y)\d y\d s=\int_{\varphi(t)}^{\varphi(T)}\int_{\varphi^{-1}(y)}^{T}\frac{\de f}{\de t}(s,y)\d s\d y\\
	&=\int_{\varphi(t)}^{\varphi(T)}f(T,y)\d y-\int_{\varphi(t)}^{\varphi(T)}f(\varphi^{-1}(y),y)\d y=\int_{\varphi(t)}^{\varphi(T)}f(T,y)\d y-\int_{t}^{T}f(s,\varphi(s))\dot{\varphi}(s)\d s.
	\end{align*}
	All the equalities are justified by part 1), Lemma~\ref{areaformula} and Corollary \ref{nullsets}.
\end{proof}
\begin{rmk}
	We can replace assumption a) in Theorem~\ref{differentiation} by the weaker
	\begin{itemize}
		\item[a$'$)] $f(T,\cdot)\in L^1(a,\varphi(T))$.
	\end{itemize}
	Indeed exploiting b) and c) one can recover a) from a$'$).
\end{rmk}
\begin{rmk}\label{remksobolev}
	If for some $p\in[1,+\infty]$ the function $f$ in Theorem~\ref{differentiation} satisfies
	\begin{itemize}
		\item[$\alpha$)]$f(t,\cdot)\in L^p(a,\varphi(t))$  for every $t\in[0,T]$,
		\item[$\beta$)] $f(\cdot,y)\in W^{1,p}(I_y)$ for a.e. $y\in[a,\varphi(T)]$, 
		\item[$\gamma$)] $\displaystyle\frac{\de f}{\de t}\in L^p(\Omega_T^\varphi)$,
	\end{itemize}
	then the function $F$ belongs to $W^{1,p}(0,T)$ and the same formula for the derivative holds. As in Remark~\ref{sobolev}, for the validity of this fact we cannot replace $\varphi\in C^{0,1}([a,b])$ by $\varphi\in AC([a,b])$.
\end{rmk}

\section*{Acknowledgements}
The authors wish to thank Prof. Gianni Dal Maso and Giuliano Lazzaroni for many helpful discussions on the topic. A special thank goes also to an anonymous reviewer for the correct interpretation of the model as a horizontally loaded bar. This material is based on work supported by the INdAM-GNAMPA Project 2018 ``Analisi variazionale per difetti e interfacce nei materiali''.

	\bigskip
	
	\bibliographystyle{siam}

\end{document}